\documentclass[preprint,11pt,authoryear]{elsarticle}
\usepackage{tikz}
\usepackage{xcolor}
\usepackage{amsfonts}
\usepackage{amsmath}
\usepackage{amsthm}
\usepackage{amssymb}
\usepackage{times}
\usepackage{subfig}
\usepackage{multirow}
\usepackage{setspace}
\usepackage{enumerate}
\usepackage{natbib}
\usepackage{color}
\usepackage{colortbl}
\usepackage{ wasysym }
\usepackage{verbatim}
 \usepackage{lscape}
\usepackage{array}
\usepackage{setspace}

\usepackage{caption}
\usepackage{rotating}

\newcolumntype{x}[1]{%
>{\centering\hspace{0pt}}p{#1}}%

\def\i.i.d.{\buildrel {\rm i.i.d.} \over \sim}

\def\cw#1 { \overset{\mathbb{P}}{\underset{#1}{\longrightarrow}} }

\def\P#1{{\mathbb{P}}\left(#1\right)}

\def\E#1{{\mathbb E}\left[#1\right]}

\def\Var#1{{\mathrm Var}\left(#1\right)}

\def \rcov#1#2 {{\rm cov}_{#1}\left( #2\right)}

\newtheorem{lemma}{Lemma}
\newtheorem{theorem}{Theorem}

\newtheorem{proposition}{Proposition}

\newtheorem*{toy*}{Toy Model}

\newtheorem{toymodel}{Alternative Model}

\newtheorem{model}{Model}

\def\cov#1{{\rm  cov}\left[#1\right]}

\begin{document}
\begin{frontmatter}
\title{The Poisson random effect model for experience ratemaking: limitations and alternative solutions}

\author[IH]{Woojoo Lee\corref{cor1}}
\ead{lwj221@gmail.com}
\address[IH]{Department of Statistics, Inha University, 235 Yonghyun-Dong, Nam-Gu, Incheon 402-751, Korea.}
\author[IH]{Jeonghwan Kim\corref{cor1}}
\ead{sinkei9456@naver.com}

\author[EH]{Jae Youn Ahn\corref{cor2}}
\ead{jaeyahn@ewha.ac.kr}
\address[EH]{Department of Statistics, Ewha Womans University, 11-1 Daehyun-Dong, Seodaemun-Gu, Seoul 120-750, Korea.}
\cortext[cor1]{First Authors}
\cortext[cor2]{Corresponding Author}

\begin{abstract}
Poisson random effect models with a shared random effect have been widely used in actuarial science for analyzing the number of claims. In particular, the random effect is a key factor in \textit{a posteriori} risk classification. However, the necessity of the random effect may not be properly assessed due to the dual role of the random effect; it affects both the marginal distribution of the number of claims and the dependence among the numbers of
claims obtained from an individual over time.
In line with such observations, we explain that one should be careful in using the score test for the nullity of the variance of the shared random effect, as a sufficient condition for the existence of the posteriori risk classification. 
To safely perform the \textit{a posteriori} risk classification,
we propose considering an alternative random effect model based on the negative binomial distribution, and
show that safer conclusions about the \textit{a posteriori} risk classification can be made based on it.
We also derive the score test as a sufficient condition for the existence of the \textit{a posteriori} risk classification based on the proposed model.
\end{abstract}


\end{frontmatter}

\vfill

\pagebreak

\vfill

\pagebreak

\section{Introduction}

Ratemaking is a key process in pricing risks in actuarial science.
Risk classification enables ratemaking
by grouping the insured into several homogeneous groups in terms of degree of risks.
Following the terminology in \cite{Pin97}, \cite{Den07} and \cite{Ant12}, the risk classification procedure is divided into \textit{a priori} and \textit{a posteriori} risk classifications.
\textit{A priori} risk classification uses
the measured information of the policyholder, such as gender and age, which are available at the moment of contract.
A standard statistical tool for \textit{a priori} risk classification is the generalized linear model (GLM).
However,
hidden characteristics such as driving skill and
knowledge of road conditions have an impact on risks of the policyholder; therefore, these become sources of heterogeneity in a portfolio.
To reflect these unobserved risk characteristics of policyholders,
random effects are added to GLMs. A standard tool for \textit{a posteriori} risk classification
is the generalized linear mixed model (GLMM).
A book-length review for these models is given in \cite{DeJong08} and \cite{Frees10}.
In this study, our concern is the analysis of the number of claims applicable to automobile insurance.
Specifically, we focus on a Poisson random effect model because it has been widely used in actuarial science as a claim frequency model for
\textit{a posteriori} risk classification. A good review of actuarial modeling of claim counts based on the Poisson or its variations
distribution are given in \cite{Yip2005modeling}, \cite{Den07},  and \cite{Boucher2008credibility}.

\cite{Pin97} proposed an analysis pipeline for the number of claims based on a Poisson random effect model, where
the random effect denotes the heterogeneity component for each policyholder.
The first step is to perform a score test for the necessity of the random effect with a fitted Poisson GLM, that is, whether the random effect variance is zero or not.
When the nullity of the random effect variance is rejected, \cite{Pin97} performed \textit{a posteriori} ratemaking by estimating the random effect. This procedure is simple to implement,
and intuitively appealing. However, in this study, we emphasize that this approach should be used carefully because it does not
contemplate the dual role of the random effect in the Poisson model.
Introducing the random effect changes both the marginal distribution for the number of claims and the dependence among the numbers of claims obtained from an
individual over time. \cite{Den07} and \cite{Mur13} also highlight this, stating that
the random effect induces both overdispersion and serial dependence.
This implies that the score test detects not only changes in the marginal distribution, but also serial dependence.
In other words,
the testing procedure proposed in \cite{Pin97} can reject the null hypothesis even when the data show pure overdispersion
without serial dependence. This invokes an important problem in experience ratemaking because
falsely detected serial dependence may spoil the fairness of the existing rating system.
In this study, we first prove that the serial dependence is falsely detected with probability one
when the score test is applied to independent, but overdispersed data.
This implies that one should be careful in using the score test as a sufficient condition for the existence of the bonus-malus (BM) system.
To mitigate this danger,
we consider an alternative random effect model based on the negative binomial distribution. In addition,
based on the proposed model, we develop a new score test for checking the necessity of the heterogeneity component for each policyholder.

The remainder of this paper is organized as follows.
Section 2 reviews how experience ratemaking is performed with a Poisson random effect model.
Section 3 explains possible dangers when the Poisson random effect model is used and
asymptotic results are provided. We suggest an alternative random effect model to overcome the limitation in Section 4, and
provide a new score test for the necessity of the heterogeneity component for each policyholder, followed by
a numerical study in Section 5. Section 6 explains how \textit{a posteriori} ratemaking is performed with with the proposed model and compares the performance of the \textit{a {\it posteriori}} risk classification with the Poisson random effect model,
followed by concluding remarks in Section 7.


\section{Review of experience ratemaking for claim frequency}

Let $Y\sim {\rm Gamma}(a,b)$ be a gamma distribution with mean $a/b$ and variance $a/b^2$, where $a$ and $b$ are called shape and rate parameters, respectively.
We also denote $N\sim NB(\lambda, \alpha)$ as a negative binomial distribution with mean $\lambda$ and variance $\lambda+\alpha\lambda^2$, where $1/\alpha$ is called a dispersion parameter. Denote $Y\sim {\rm Lognormal}(u, \sigma^2)$ as log-normal distribution with mean $\exp(u+\sigma^2/2)$ and variance
$\left[ \exp(\sigma^2)-1 \right] \exp(2u+\sigma^2)$.

Following \cite{Pin97} and \cite{Pin98}, we consider
\begin{equation} \label{Pin0}
N_{it}|\theta_{i} \sim {\rm Pois}(\lambda_{it}\theta_i), \quad i=1,\ldots,k \quad \hbox{and} \quad t=1,\ldots,T_{i}
\end{equation}
where
\begin{equation}\label{mu1}
\lambda_{it}:=\exp(x^{T}_{it}{\boldsymbol{\beta}})
\end{equation}
and $N_{it}$ is the number of claims reported by the $i$-th policyholder in
period $t$.

Here, ${\boldsymbol x}_{it}$  is a covariate vector and $\theta_{i}$
denotes the random effect for the $i$-th policyholder to explain the heterogeneity component.
The random effect $\theta_{i}$ is assumed to follow $\rm Gamma(a,a)$ where $E(\theta_{i})=1$ and $Var(\theta_{i})=1/a$, which
leads to a negative binomial distribution for the marginal distribution of $N_{it}$.
This random intercept model has been widely used in actuarial science \citep{Bou06}.
However, the gamma distribution for $\theta_{i}$ is not compulsory. Other distributions such as
inverse Gaussian or log-normal distributions can be used. For the details, see \cite{Bou06}.

\cite{Pin97} derived a sufficient condition for the existence of a BM system from the random effect model (\ref{Pin0})
without any parametric assumption for $\theta_{i}$ and showed that checking the condition is equivalent to performing a score test
for $H_{0}:Var(\theta_{i})=0$. The analytic form for the score statistic and its asymptotic distribution is given as
\begin{equation}\label{Pin}
{\mathcal T}_{\rm pin}=\frac{\sum_{i}\left(\left(\sum_{t}(N_{it}-\widehat{\lambda}_{it})\right)^2 -\sum_{t} N_{it}\right) }{\sqrt{2\sum_{i}\left(\sum_{t}\widehat{\lambda}_{it}\right)^2}}\sim N(0,1),
\end{equation}
where $\widehat{\lambda}_{it}=\exp(x^{T}_{it}\widehat{{\boldsymbol{\beta}}}_{0})$ and $\widehat{{\boldsymbol{\beta}}}_{0}$ denotes the maximum likelihood estimate for ${\boldsymbol{\beta}}$ under $H_{0}$.
The null hypothesis is rejected when ${\mathcal T}_{\rm pin}$ is greater than or equal to the $1-\alpha$ quantile of $N(0,1)$.
When the rejection occurs, \cite{Pin97} used the following BM coefficient for experience ratemaking:
$$\E{\theta_{i}|x_{i1},\ldots,x_{iT_{i}},N_{i1},\ldots,N_{iT_{i}}}=\frac{a+\sum_{t}N_{it}}{a+\sum_{t}\lambda_{it}}$$
where $\theta_{i}$ was assumed to follow $Gamma(a,a)$.


%
%
%

\section{Problem of BM system based on the Poisson random effect model}\label{sec.mot}

Although the score statistic \eqref{Pin} is useful in investigating $H_{0}:Var(\theta_{i})=0$, caution is necessary when it involves the condition for the existence of a BM system, as Theorem \ref{theo1} shows. We define related models before presenting the theorem.

\begin{model}[No Random Effect Model]\label{mo.1}
   For $i=1, \cdots, k$ and $t=1, \cdots, T_i$, consider the following generalized linear model
  $$N_{it}\sim {\rm Pois}(\lambda_{it})$$
 where $\lambda_{it}$ is defined in \eqref{mu1}.
\end{model}

\begin{model}[Shared Random Effect Model]\label{mo.2}
For $i=1, \cdots, k$ and $t=1, \cdots, T_i$, consider the following random effect model
$$
N_{it}\big\vert \theta_i\sim {\rm Pois}(\theta_i\lambda_{it})
$$
where $\lambda_{it}$ is defined in \eqref{mu1} and
$\theta_i$ are i.i.d with mean 1 and variance $\sigma^2$.
\end{model}

  Under Model \ref{mo.2}, we have
  \begin{equation}\label{eq.33}
  \E{N_{it}}=\lambda_{it} \quad \hbox{and}\quad \Var{N_{it}}=\lambda_{it}+\lambda_{it}^2 \sigma^2 \quad\hbox{for any}\quad i\quad \hbox{and}\quad t
  \end{equation}
  and
  \begin{equation}\label{eq.3}
  \cov{N_{i t_1}, N_{i t_2}}=\lambda_{it_{1}}\lambda_{it_{2}}\sigma^2 \quad\hbox{for any}\quad i\quad \hbox{and}\quad t_1\neq t_2.
  \end{equation}
  Frequencies from the same policyholder are correlated through the shared random effect $\theta_i$.
  Here, we note that the condition for the existence of a BM system is
  \begin{equation}\label{eq.1}
  \cov{N_{i t_1}, N_{i t_2}}>0,
  \end{equation}
  and the condition for overdispersion is
  \begin{equation}\label{eq.2}
  \Var{N_{it}}>\E{N_{it}}.
  \end{equation}
  Comparing Model \ref{mo.1} and \ref{mo.2}, under the assumption $\lambda_{it}>0$, both the conditions in \eqref{eq.1} and \eqref{eq.2} are equivalent with $\sigma^2 >0$.
  Hence, as shown in \cite{Pin97}, the score statistic \eqref{Pin} for testing overdispersion determines
  the condition for the existence of a BM system under Model \ref{mo.2}. However, considering the following model, it is clear that the two conditions in \eqref{eq.1} and \eqref{eq.2} are not always equivalent.

\begin{toymodel}[Saturated Random Effect Model]\label{mo.3}
For $i=1, \cdots, k$ and $t=1, \cdots, T_i$, consider the following random effect model
$$N_{it}\big\vert \theta_{it}\sim {\rm Pois}(\theta_{it}\lambda_{it})$$  where
$\lambda_{it}$ is defined in \eqref{mu1} and
$\theta_{it}$ are i.i.d with mean 1 and variance $\tau^2$.
\end{toymodel}

To distinguish $\theta_{it}$ from $\theta_{i}$, we call $\theta_{i}$ and $\theta_{it}$ as {\it shared random effect} and {\it saturated random effect}, respectively.
Under the Alternative Model \ref{mo.3}, we have
  \[
  \E{N_{it}}=\lambda_{it} \quad \hbox{and}\quad \Var{N_{it}}=\lambda_{it}+\lambda_{it}^2 \tau^2, \quad i=1, \cdots, k\quad \hbox{and}\quad t=1, \cdots, T_i
  \]
  and
  \begin{equation}\label{eq.4}
  \cov{N_{i t_1}, N_{i t_2}}=0, \quad\hbox{for any}\quad i\quad \hbox{and}\quad t_1\neq t_2.
  \end{equation}
Hence, the Alternative Model \ref{mo.3} does not require a BM system because there is no correlation between the frequencies from the same policyholder.
However, as will be proved later, the score statistic \eqref{Pin} rejects the null hypothesis with
probability one, as $k$ goes to infinity under the Alternative Model \ref{mo.3}, which may result in erroneous \textit{a posteriori} experience ratemaking.
Alternative Model \ref{mo.3} highlights that the necessity of the random effect does not always support the use of the BM system.

\subsection{Asymptotic result of the score statistic}

An asymptotic property of the score statistic \eqref{Pin} is proved below.
For simplicity, we assume that $\max_{i} T_{i}$ is bounded.

\begin{theorem} \label{theo1}
Let $\widehat{{\boldsymbol{\beta}}}$ be the maximum likelihood estimator for the Poisson regression model (\ref{mo.1})
and $\widehat{\lambda}_{it}=\exp(x^{T}_{it}\widehat{{\boldsymbol{\beta}}})$.
Assume that $\sum_{i}(\sum_{t}\lambda_{it})^2/k$ converges to a positive constant.
Under Alternative Model \ref{mo.3}, the test statistic (\ref{Pin})
rejects the null hypothesis with probability one, as $k$ goes to infinity.
\end{theorem}

\begin{proof}
Since the Poisson regression model (\ref{mo.1}) and Alternative Model \ref{mo.3} have the same marginal mean model,
the Poisson regression model (\ref{mo.1}) provides a $\sqrt{k}$-consistent estimator for ${\boldsymbol{\beta}}$ under Alternative Model \ref{mo.3}.
Therefore, we have $\widehat{\lambda}_{it}-\lambda_{it}=O_{p}(k^{-1/2})$.
By the Cauchy-Schwartz inequality, $\sum_{i}\sum_{t}\lambda_{it}/k$ also converges to a positive constant.
Then, the score statistic becomes
\begin{eqnarray*}
\frac{\sum_{i}\left(\left(\sum_{t}(N_{it}-\widehat{\lambda}_{it})\right)^2 -\sum_{t} N_{it}\right) }{\sqrt{2\sum_{i}\left(\sum_{t}\widehat{\lambda}_{it}\right)^2}}
&=&\frac{\sum_{i}\left(\left(\sum_{t}(N_{it}-{\lambda}_{it})\right)^2 -\sum_{t} N_{it}\right) }{\sqrt{2\sum_{i}\left(\sum_{t}{\lambda}_{it}\right)^2}}+O_{p}(k^{-1/2}).
\end{eqnarray*}
Because
$E(\left(\sum_{t}(N_{it}-{\lambda}_{it})\right)^2 - \sum_{t} N_{it})=\tau^2 (\lambda_{it})^2 >0 $,
\begin{eqnarray*}
\frac{\sum_{i}\left(\left(\sum_{t}(N_{it}-{\lambda}_{it})\right)^2 -\sum_{t} N_{it}\right) }{\sqrt{2\sum_{i}\left(\sum_{t}{\lambda}_{it}\right)^2}}=O_{p}(k^{1/2})
\end{eqnarray*}
Therefore, the test statistic goes to infinity with probability one, as $k$ goes to infinity.
\end{proof}

In conclusion, although the test statistic proposed in \cite{Pin97} supports the existence of the non-degenerate random effect, one needs to distinguish between Model \ref{mo.2} and Alternative Model \ref{mo.3} to determine the dependence between $N_{it_1}$ and  $N_{it_2}$.
As the test statistic \eqref{Pin} rejects $H_{0}:Var(\theta_{i})=0$ even for independent data showing overdispersion, caution is necessary in using the score test as a sufficient condition for the existence of the BM system.

\subsection{Simulation Result}\label{toysim}

In connection with Theorem \ref{theo1}, the following simulation study shows that even for independent $N_{it}$, Model \ref{mo.2} shows a considerable amount of variance for the shared random effect $\theta_{i}$,
which results in a false correlation between frequencies from the same policyholder.
Assume that $N_{it}$ ($i=1, \cdots, 30$ and $t=1, \cdots, 5$) are obtained from Alternative Model \ref{mo.3}
with
\begin{equation}\label{eq.23}
\log(\lambda_{it})={\boldsymbol{x}}_{it}^{\mathrm T}{{\boldsymbol{\beta}}}
\end{equation}
where
\begin{equation}\label{eq.24}
{{\boldsymbol{\beta}}}^{\mathrm T}= (-0.5, 0.5, 0.5)\quad\hbox{and}\quad
{\boldsymbol{x}}_{it}^{\mathrm T}=
\begin{cases}
  (1,1, 1), & i\equiv 1\; ({\rm mod}\;6);\\
  (1,1, 2), & i\equiv 2\; ({\rm mod}\;6);\\
  (1,2, 1), & i\equiv 3\; ({\rm mod}\;6);\\
  (1,2, 2), & i\equiv 4\; ({\rm mod}\;6);\\
  (1,3, 1), & i\equiv 5\; ({\rm mod}\;6);\\
  (1,3, 2), & i\equiv 6\; ({\rm mod}\;6).\\
\end{cases}
\end{equation}
Here, the variances of the saturate random effect are set as
\[
\Var{\theta_{it}}=0, \,1/6,\, 2/6,\, 3/6.
\]
Detailed results, based on $100$ repetitions, are in Figure \ref{fig.1} and Table \ref{table.1} when Model \ref{mo.2} is fitted to this data.
Estimate for $\Var{\theta_{i}}$ increases as the true value for $\Var{\theta_{it}}$ increases.

\begin{figure}[!ht]
\centering
\includegraphics[width=0.5\textwidth]{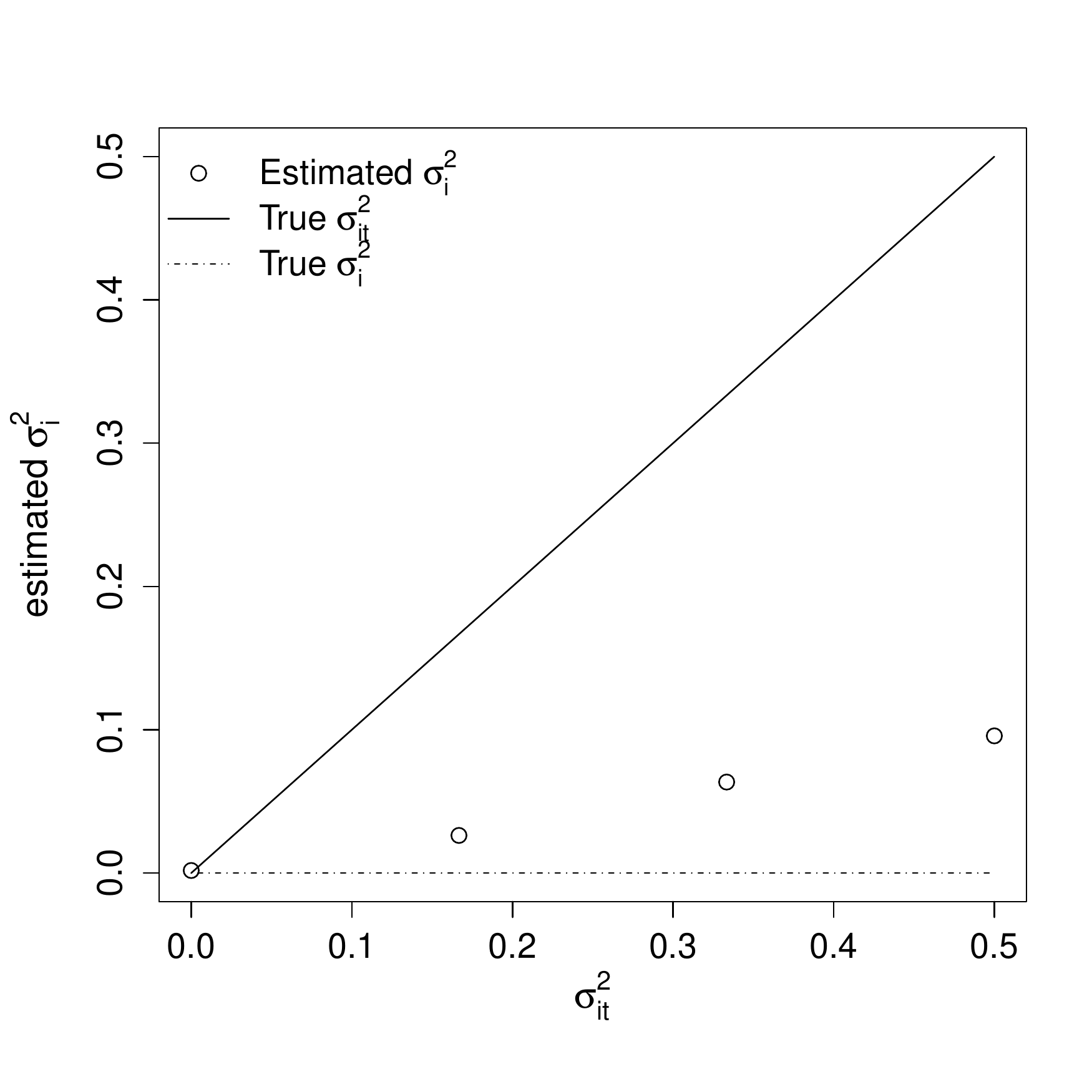}
\caption{Estimate for $Var(\theta_{i})$ under Model \ref{mo.2} when data are generated from Alternative Model \ref{mo.3}}
\label{fig.1}
\end{figure}

\begin{table}[!ht]

\centering
\begin{tabular}{|c|c|c|}
\hline
\multicolumn{2}{|c|}{True} & Estimated\\
\hline
$Var(\theta_{i})$ & $Var(\theta_{it})$ & $Var(\theta_{i})$\\
\hline
$0$ & $0$   & $0.00094\, (0.00194)$ \\
$0$ & $0.16667$ & $0.02678\, (0.01189)$ \\
$0$ & $0.33333$ & $0.06557\, (0.02476)$ \\
$0$ & $0.50000$ & $0.09888\, (0.03573)$ \\
\hline
\end{tabular}
\caption{Estimate for $Var(\theta_{i})$ under Model \ref{mo.2} when data are generated from Alternative Model \ref{mo.3}. The values in the parentheses are standard errors.}\label{table.1}
\end{table}

\section{Alternative model to overcome the limitation of the Poisson random effect model}

The limitation of Model \ref{mo.2} results from using the random effect $\theta_{i}$ for not only explaining the correlation among individuals but also for changing the marginal
distribution. To separate these two roles of $\theta_{i}$, we need to incorporate
the saturated random effect $\theta_{it}$ in Model \ref{mo.2} to solely capture overdispersion.
Therefore, the proposed random effect model for the number of claims is
\begin{equation}\label{ourmodel}
N_{it}|\theta_{i},\theta_{it} \sim {\rm Pois}(\lambda_{it}\theta_{i}\theta_{it}), \quad i=1, \cdots, k\quad \hbox{and}\quad t=1, \cdots, T_i
\end{equation}
where $\lambda_{it}$ is defined in \eqref{mu1} and $\E{\theta_{t}}=\E{\theta_{it}}=1$.
However, this model may not be preferred in terms of computation because
incorporating the saturated random effect requires inversion and multiplications of high dimensional matrices if the number of observations is large.
To make the model (\ref{ourmodel}) practically useful,
we employ a parametric assumption for $\theta_{it}$.
Our convenient choice is $\theta_{it}\sim {\rm Gamma}(1/a, 1/a)$ because
this choice allows us to avoid such large matrix computations by analytically integrating $\theta_{it}$ out of the likelihood function.

\begin{model}\label{mod.1}
For $i=1, \cdots, k$ and $t=1, \cdots, T_i$, consider
\begin{equation}\label{nbm}
N_{it}|\theta_{i} \sim {\rm NB}(\lambda_{it}\theta_i,\alpha), \quad
\end{equation}
where $\lambda_{it}$ is defined in \eqref{mu1} and $\theta_{i}$ are i.i.d with mean 1 and variance $b$.
\end{model}

Model \ref{mo.2} is a special case of Model \ref{mod.1} with $\alpha=0$.
In the numerical study in Section \ref{toysim}, we observe that Model \ref{mo.2} shows considerable shared random effect when fitted to the independent data showing overdispersion.
On the other hand, because Model \ref{mod.1} already takes care of overdispersion, we expect that Model \ref{mod.1} does not have this problem.
The following is a detailed simulation study to assess the performance of Model \ref{mo.2} and Model \ref{mod.1}.

\subsection{Simulation Study I}\label{sec.sim1}

We compare the performance of Model \ref{mo.2} and Model \ref{mod.1} to estimate the variance of $\theta_{i}$ under various simulation settings.
First, $N_{it}$ is generated from Model \ref{mod.1} with (\ref{eq.23}) and (\ref{eq.24}).
For the random effects, we assume that
\begin{equation}\label{eq.21}
\theta_{i}\sim {\rm Lognormal}(-\sigma^2/2, \sigma^2), \quad \sigma^2=0,\, 1/6,\, 2/6,\, 3/6
\end{equation}
and
\begin{equation}\label{eq.22}
\theta_{it}\sim {\rm Gamma}(1/a, 1/a), \quad a=0,\, 1/3,\, 2/3,\, 1.
\end{equation}
Here, $\sigma^2=0$ and $a=0$ represent $\P{\theta_i=1}=1$ and $\P{\theta_{it}=1}=1$, respectively.
In total, we have 16 combinations of $(\sigma^2,a)$. For each combination, the following claim frequencies
$$N_{it}, \quad i=1, \cdots, 120\quad\hbox{and}\quad t=1, \cdots, 5$$
are generated, and Model \ref{mo.2} and Model \ref{mod.1} are fitted to the simulated data.
Estimation results, based on $100$ repetitions, for the parameter $\Var{\theta_{t}}=\sigma^2$ in each setting are shown in Figure \ref{fig.2}. As expected, Model \ref{mo.2} overestimates $\sigma^2$, and the bias becomes larger as $\Var{\theta_{it}}$ becomes larger, whereas Model \ref{mod.1} shows an unbiased result for estimating $\sigma^2$, regardless of the size of $\Var{\theta_{it}}$.

\subsection{Simulation Study II}\label{sec.sim2}

Although the NB distribution is a flexible model that can explain extra-Poisson variation,
it does not capture all types of extra-Poisson variation. Therefore, our concern is
the performance of Model \ref{mod.1} when there is overdispersion effect beyond the negative binomial distribution.
Specifically, we
consider the following random effect model, which has an additional independent random effect on Model \ref{mod.1}:
\begin{toymodel} \label{mod.4}
\begin{eqnarray}
N_{it}|\theta_{i}, \theta_{it} \sim {\rm NB}(\lambda_{it}\theta_i\theta_{it} ,\alpha).
\end{eqnarray}
\end{toymodel}
We generate the claim frequencies
$$ N_{it}, \quad i=1, \cdots, 120\quad\hbox{and}\quad t=1, \cdots, 5 $$
from Alternative Model \ref{mod.4}. Then, Model \ref{mo.2} and Model \ref{mod.1} are fitted to this simulated data for comparison.
The parameter settings for fixed and random effects are the same as in Section \ref{sec.sim1}.
For each of the $16$ simulation settings, estimation results, based on $100$ repetitions, for the parameter $\Var{\theta_{i}}=\sigma^2$ are in Figure \ref{fig.3}.
As expected, Model \ref{mo.2} shows severe upward bias, as the variance of $\theta_{it}$ becomes larger, whereas
Model \ref{mod.1} shows much smaller bias; however, the amount of bias increases as the size of $\Var{\theta_{it}}$ increases.
From this observation, the dispersion parameter $\alpha$ in the negative binomial distribution is effective in reducing the bias in the estimate of $\sigma^2$. In fact, we can obtain an unbiased estimate for $\sigma^2$ by fitting
Alternative Model \ref{mod.4} directly; however, as its use is often prohibited owing to heavy computational burden, we do not consider this model.

\section{Score test for checking the existence of a BM system}
In this section, we derive a sufficient condition for the existence of a BM system and perform a numerical study to check the performance of the proposed score test.
First, similar to \cite{Pin97}, we propose using the score test for $H_{0}:Var(\theta_{i})=0$ under Model \ref{mod.1}.
The following theorem gives an analytic form of the score statistic and its asymptotic distribution.


\begin{theorem} \label{theo2}
Assume that
$$\lim\limits_{k\rightarrow\infty}\max\limits_{i\in \{1, 2, \cdots, k\}}  T_{i}$$ is bounded.
Let ${\boldsymbol \omega}=({\boldsymbol{\beta}},\alpha)$ and $\widehat{{\boldsymbol{\beta}}}$ and $\widehat{\alpha}$ denote
the maximum likelihood estimates from Model \ref{mod.1} under $H_{0}:Var(\theta_{i})=0$. Under $H_{0}$,
the score statistic and its asymptotic distribution is given as
$$\sum_{i}{\mathcal T}_i(\widehat{{\boldsymbol{\beta}}},\widehat{\alpha})/\sqrt{\widehat{I}_{\sigma^2\sigma^2}-\widehat{I}_{\sigma^2{\boldsymbol \omega}}\widehat{I}^{-1}_{{\boldsymbol \omega}{\boldsymbol \omega}}\widehat{I}^{T}_{\sigma^2{\boldsymbol \omega}}}\sim N(0,1)$$
where
$${\mathcal T}_i(\widehat{{\boldsymbol{\beta}}},\widehat{\alpha}) :=
\frac{1}{2} \left( (\sum_{t} \frac{N_{it}-\widehat{\lambda}_{it}}{1+\widehat{\alpha}\widehat{\lambda}_{it}})^2 -\sum_{t}\frac{N_{it}(1+\widehat{\alpha}\widehat{\lambda}_{it})^2-\widehat{\alpha}^2 N_{it}\widehat{\lambda}_{it}^2 - \widehat{\alpha}\widehat{\lambda}_{it}^2}{(1+\widehat{\alpha}\widehat{\lambda}_{it})^2}\right)$$
and
the analytic forms on $\widehat{I}_{\sigma^2\sigma^2}$, $\widehat{I}_{\sigma^2{\boldsymbol \omega}}$ and
$\widehat{I}_{{\boldsymbol \omega}{\boldsymbol \omega}}$ are given in the proof.
\end{theorem}
\begin{proof}
Let $\ell_{i}=\ell(n_{it}|v_{i})$ denote the log negative binomial likelihood function, given the random effect.
Following \cite{Che84} and \cite{Lia87}, the score function for $Var(\theta_{i})$ under $H_{0}$ is
\begin{eqnarray*}
{\mathcal T}_i({{\boldsymbol{\beta}}},{\alpha})&=& \frac{1}{2} \sum_{i} \left[ \left( \left. \frac{\partial{\log f(y_{i}|v_{i})}}{\partial{v_{i}}} \right|_{v_{i}=1} \right)^{2}
                    + \left( \left. \frac{\partial^{2}{\log f(y_{i}|v_{i})}}{\partial{v_{i}^{2}}} \right|_{v_{i}=1} \right)
                      \right]\\
                      &=&\frac{1}{2} \left( (\sum_{t} \frac{N_{it}-{\lambda}_{it}}{1+{\alpha}{\lambda}_{it}})^2 -\sum_{t}\frac{N_{it}(1+{\alpha}{\lambda}_{it})^2-{\alpha}^2 N_{it}{\lambda}_{it}^2 - {\alpha}{\lambda}_{it}^2}{(1+{\alpha}{\lambda}_{it})^2}\right).
\end{eqnarray*}
By taking into account the uncertainty of $\widehat{{\boldsymbol{\beta}}}$ and $\widehat{\alpha}$ in the asymptotic variance of ${\mathcal T}_i(\widehat{{\boldsymbol{\beta}}},\widehat{\alpha})$, the following quantities become necessary:
\begin{eqnarray*}
\widehat{I}_{\sigma^2\sigma^2}&=&\sum_{i}\E{\left(\frac{\partial \ell_{i}}{\partial \sigma^2}\right)^2}
= \frac{1}{4} \sum_{i} \left( \sum_{t} \frac{2\widehat{\lambda}_{it}^2(1+\widehat{\alpha})}{(1+\widehat{\alpha}\widehat{\lambda}_{it})^2} + 4 \sum_{t<t'} \frac{\widehat{\lambda}_{it}(1+\widehat{\alpha}\widehat{\lambda}_{it})}{(1+\widehat{\alpha}\widehat{\lambda}_{it})^2} \frac{\widehat{\lambda}_{it'}(1+\widehat{\alpha}\widehat{\lambda}_{it'})}{(1+\widehat{\alpha}\widehat{\lambda}_{it'})^2} \right) \\
\widehat{I}_{\sigma^2{\boldsymbol \omega}}&=&\left(\sum_{i}\E{\frac{\partial \ell_{i}}{\partial \sigma^2}\frac{\partial \ell_{i}}{\partial {\boldsymbol{\beta}}}}, \sum_{i}\E{\frac{\partial \ell_{i}}{\partial \sigma^2}\frac{\partial \ell_{i}}{\partial \alpha}}\right) =
\left(0, \frac{1}{2} \sum_{i} \left( \sum_{t} \frac{\widehat{\lambda}_{it}^2}{(1+\widehat{\alpha}\widehat{\lambda}_{it})^2} \right)\right) \\
\widehat{I}_{{\boldsymbol \omega}{\boldsymbol \omega}}&=&\left(
\begin{array}{cc}
  \sum_{i}\E{\left(\frac{\partial \ell_{i}}{\partial {\boldsymbol{\beta}}}\right)^2} &
  \sum_{i}\E{\frac{\partial \ell_{i}}{\partial {\boldsymbol{\beta}}}\frac{\partial \ell_{i}}{\partial \alpha}} \\
  \sum_{i}\E{\frac{\partial \ell_{i}}{\partial \alpha}\frac{\partial \ell_{i}}{\partial {\boldsymbol{\beta}}}} &
  \sum_{i}\E{\left(\frac{\partial \ell_{i}}{\partial \alpha}\right)^2}
\end{array} \right) = \left(
\begin{array}{cc}
  X^{T} W X & 0 \\
 0 &   i(\widehat{{\boldsymbol{\beta}}}, \widehat{\alpha})
\end{array} \right)
\end{eqnarray*}
Here, $X=(x_{1},x_{2},\ldots,x_{n})^{T}$ and
$W$ is a block diagonal matrix expressed as
\begin{eqnarray*}
W = \left( \begin{array}{ccccc}
  W_1 & 0   & 0   & \ldots & 0   \\
  0   & \ldots & 0   & \ldots & 0   \\
  0   & 0   & W_i & \ldots & 0   \\
  0   & 0   & 0   & \ldots & 0   \\
  0   & 0   & 0   & \ldots & W_k \\
\end{array} \right)
\end{eqnarray*}
Where the $t$-th element of $W_i$ is $\widehat{\lambda}_{it} / (1 + \widehat{\alpha} \widehat{\lambda}_{it})$.
The $i$-th term of $i(\widehat{{\boldsymbol{\beta}}}, \widehat{\alpha})$ is equal to
$$\sum_t \left( \widehat{\alpha}^{-4} \left(
\sum_{j=0}^{\infty}
\left(\widehat{\alpha}^{-1} + j\right)^{-2} \Pr(N_{it} \geq j + 1) -
\frac{\widehat{\alpha} \widehat{\lambda}_{it}}{\widehat{\lambda}_{it} + \widehat{\alpha}^{-1}}\right)
\right).$$
All the expectations are evaluated under $H_{0}$, and ${\boldsymbol{\beta}}$ and $\alpha$ are replaced with their MLEs $\widehat{{\boldsymbol{\beta}}}$ and $\widehat{\alpha}$.
The details for the above expectations are given in Lemma \ref{appen.lem} in the Appendix. Based on \cite{Che84}'s result, the score statistic follows asymptotically $N(0,1)$.

\end{proof}

As this score statistic requires only MLEs under $H_{0}$, it is sufficient for fitting an ordinary NB regression model to perform the test.

\subsection{Simulation Study}\label{sec.simNBscore}

We perform a numerical study to check the performance of the score test in terms of empirical type I error and power under Model \ref{mod.1}.
We generate
$$ N_{it}, \quad i=1, \cdots, 100\quad\hbox{and}\quad t=1, \cdots, 5 $$
from
${\rm NB}(\lambda_{it}\theta_i ,\alpha)$ with $\lambda_{it}=\exp({{\beta}}_{0}+{{\beta}}_{1}x_{ij})$ and $\theta_{i}\sim {\rm Gamma}(1/\sigma^2, 1/\sigma^2)$.
Therefore, $\E{\theta_{i}}=1$ and ${\rm Var}(\theta_{i})=\sigma^2$.
For the variance of random effect $\theta_i$, we consider $\sigma^2=0, 0.1, 0.2, \ldots, 1$. The $x_{ij}$ are generated from $U(0,1)$, and ${{\beta}}_{0}=0$ and ${{\beta}}_{1}=1$.
For the parameter $\alpha$ in the negative binomial distribution, $0.2, 0.5,$ and $1.0$ are considered.
We report the proportion of rejecting $H_{0}$ based on 1000 replications.
At $\sigma^2 =0$, this proportion corresponds to the type I error, and
at nonzero $\sigma^2$, it denotes the empirical power.
The results are plotted in Figure \ref{fig.3-2}.
As similar results are observed under other settings, we omit them for convenience.
In general, the proposed score test controls the type I error well at its nominal
Level of $0.05$ and shows that its power increases with $\sigma$.

  \begin{figure}[!ht]
     \centering
      \includegraphics[width=0.60\textwidth]{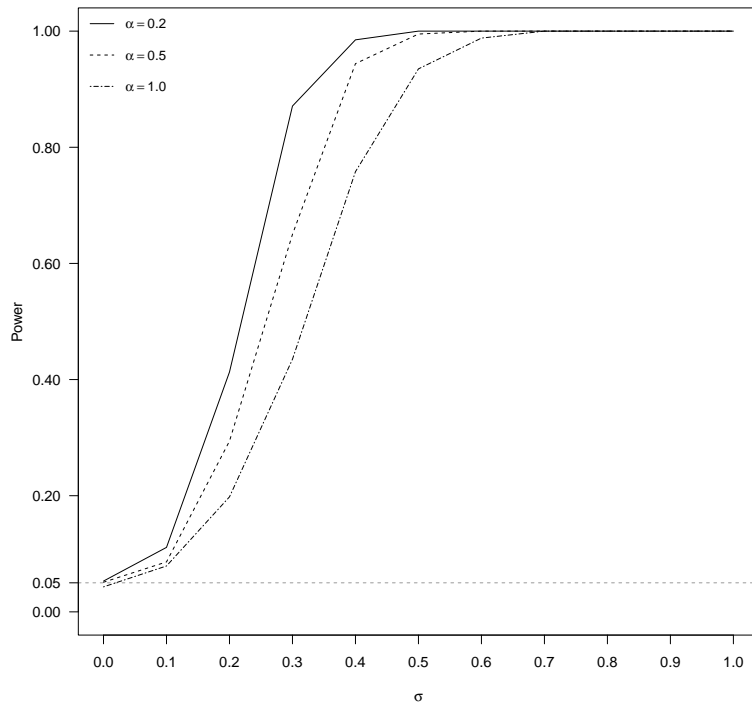}
      \caption{Empirical Type I error and power of the score test for Model \ref{mod.1}.}
     \label{fig.3-2}
  \end{figure}

\section{Experience ratemaking with B\"uhlmann Method}


For a fair valuation of premium, insurers are interested in the following predictive distribution:
\[
N_{i,t+1}\big\vert N_{i,1}, \cdots, N_{i,t}
\]
and especially, in the mean of predictive distribution
\begin{equation}\label{buhl.pred}
\E{N_{it'}\big\vert N_{i1}, \cdots, N_{iT_i}}.
\end{equation}

While the exact estimation of \eqref{buhl.pred} is possible in some special cases, it is a common practice in insurance to use the B\"uhlmann method for estimating \eqref{buhl.pred}.

%
%

Define
\[
\mu_i:=\E{\E{N_{it}\big\vert \theta_i}}, \quad \nu_i:=\E{\Var{N_{it}\big\vert \theta_i}}, \quad\hbox{and}\quad
a_i:=\Var{\E{N_{it}\big\vert \theta_i}}.
\]
Then, the B\"uhlmann factor $Z_i$ and B\"uhlmann prediction $P_i$ of the $i$-th individual are calculated as
\[
Z_i:=\frac{T_i }{\nu_i/a_i + T_i}
\]
and
\begin{equation*}
P_i:=Z_i\overline{X_i}+(1-Z_i)\mu_i,
\end{equation*}
respectively. For the details of the B\"uhlmann method, we refer to \cite{Buhlmann2006}.
The performance of the B\"uhlmann prediction is determined by the accurate estimation of the B\"uhlmann factor.
The following propositions provide the closed form expression of $\mu_i$, $\nu_i$, and $a_i$ in Model \ref{mo.2}, Model \ref{mod.1} and Alternative Model \ref{mod.4}, respectively.

\begin{proposition}\label{prop.1}
In Model \ref{mo.2}, we have
\[
\E{N_{it}\big\vert \theta_i}=\theta_i\lambda_{it} \quad\hbox{and}\quad \Var{N_{it}\big\vert \theta_i}=\theta_i\lambda_{it}
\]
which imply
\[
\begin{aligned}
\mu_i=\lambda_{it}
\end{aligned}, \quad
\begin{aligned}
\nu_i=\mu_i
\end{aligned},
\quad\hbox{and}\quad
\begin{aligned}
a_i=\mu_i^2(\exp(\sigma^2)-1).
\end{aligned}
\]
Note that $\mu_i=\lambda_{it} \exp(\sigma^2/2)$ if $\theta_{i}\sim \rm {Lognormal}(-\sigma^2/2, \sigma^2)$.
\end{proposition}

\begin{proposition}\label{prop.2}
In Model \ref{mod.1}, we have
\[
\E{N_{it}\big\vert \theta_i}=\theta_i\lambda_{it} \quad\hbox{and}\quad \Var{N_{it}\big\vert \theta_i}=\theta_i\lambda_{it} + \theta_i^2\lambda_{it}^2\alpha
\]
which imply
\[
\begin{aligned}
\mu_i=\lambda_{it}
\end{aligned}, \quad
\begin{aligned}
\nu_i=\mu_i + \alpha\mu_i^2\exp(\sigma^2)
\end{aligned},
\quad\hbox{and}\quad
\begin{aligned}
a_i=\mu_i^2(\exp(\sigma^2)-1).
\end{aligned}
\]
Note that $\mu_i=\lambda_{it} \exp(\sigma^2/2)$ if $\theta_{i}\sim \rm {Lognormal}(-\sigma^2/2, \sigma^2)$.
\end{proposition}

\begin{proposition}\label{prop.3}
In Alternative Model \ref{mod.4}, we have
\[
\E{N_{it}\big\vert \theta_i}=\theta_i\lambda_{it}
\]
and
\[
\begin{aligned}
  \Var{N_{it}\big\vert \theta_i}&= \E{\Var{N_{it}\big\vert \theta_{i}, \theta_{it}} \big\vert \theta_{i}} + \Var{\E{N_{it}\big\vert \theta_{i}, \theta_{it}} \big\vert \theta_{i}} \\
  &=\theta_i\lambda_{it} + \theta_i^2\lambda_{it}^2\alpha (b+1) \\
\end{aligned}
\]
which imply
\[
\begin{aligned}
\mu_i=\lambda_{it}
\end{aligned}, \quad
\begin{aligned}
\nu_i=\mu_i + \alpha(b+1)\mu_i^2\exp(\sigma^2)
\end{aligned},
\quad\hbox{and}\quad
\begin{aligned}
a_i=\mu_i^2(\exp(\sigma^2)-1).
\end{aligned}
\]
Note that $\mu_i=\lambda_{it} \exp(\sigma^2/2)$ if $\theta_{i}\sim \rm {Lognormal}(-\sigma^2/2, \sigma^2)$.
\end{proposition}

Assume that $N_{it}$ are generated from Model \ref{mod.1}.
If Model \ref{mo.2} is fitted to this, as shown in Proposition \ref{prop.1}, $a_i$ will be overestimated due to the overestimation of $\Var{\theta_i}$.
Knowing that B\"uhlmann premium is the linear estimator that minimizes the mean squared error(MSE), the bias in the B\"uhlmann factor leads to increase in the mean squared error.
In the following simulation studies, we compare the MSE of B\"uhlmann estimators from Model \ref{mo.2} and Model \ref{mod.1} under various settings described in Section \ref{sec.sim1} and \ref{sec.sim2}.

\subsection{Simulation Study I}

Under the same simulation setting in Section \ref{sec.sim1},
we compare the performance of the two B\"uhlmann factors using Proposition \ref{prop.1} and \ref{prop.2}.
The numerical result is summarized in Figure \ref{fig.4}. For comparison,
the B\"uhlmann factor calculated in Proposition \ref{prop.2} with the true parameter values is also plotted as the dotted line in each setting.
Model \ref{mo.2} overestimates the B\"uhlmann factor in every setting except the null saturated random effect case, that is, $\Var{\theta_{it}}=0$, whereas
Model \ref{mod.1} estimates the B\"uhlmann factor well.
We also calculate the predictive mean square error, which is defined as
\[
\E{(N_{i,t+1}-\lambda_{i,t})^2 \big\vert N_{i,1}, \cdots, N_{i,t}}.
\]
The numerical result is presented in Figure \ref{fig.6}. In each setting, the predictive MSE of Model \ref{mod.1} is smaller than that of Model \ref{mo.2}, and their difference
becomes larger as $\Var{\theta_{it}}$ increases.

\subsection{Simulation Study II}

Under the same simulation setting in Section \ref{sec.sim2}, we compare the performance of the two B\"uhlmann factors
using Proposition \ref{prop.1} and \ref{prop.2}.
The numerical result is summarized in Figure \ref{fig.5}. For comparison,
the B\"uhlmann factor calculated in Proposition \ref{prop.3} with the true parameter values is also plotted as the dotted line in each setting.
The bias of the B\"uhlmann factor calculated under Model \ref{mo.2} becomes larger than that of simulation study I, whereas
Model \ref{mod.1} shows negligible bias of the B\"uhlmann factor.
Figure \ref{fig.7} shows the predictive MSE. In each setting, the predictive MSE of Model \ref{mod.1} is smaller than that of Model \ref{mo.2}, and their difference becomes larger as $\Var{\theta_{it}}$ increases.

%


%
%


\section{Discussion}

Poisson random effect models with only a shared random effect can have a serious weakness, as a claim frequency model for \textit{a posteriori}
risk classification. In particular, one should be careful in using the score test by \cite{Pin97} as a sufficient condition for the existence of the BM system.
To prevent misleading experience ratemaking, it is necessary to incorporate shared and saturated random effects together in the Poisson random effect model.
However, using the saturated random effect is often prohibited because of heavy computational problems for large sample sizes.
Therefore, we argue that a random effect model based on NB distribution can be used to resolve the computation problem and deal with both shared and saturated random effects, demonstrating that it is possible to make safer conclusions about experience ratemaking based on it. However, we do not claim that the NB model is a flawless claim frequency model.
Extending the explanation for the weakness of the Poisson random effect model with only a shared random effect,
the proposed NB random effect model may have a similar weakness if the true model is an NB model with both shared and saturated random effects, although
the current numerical study shows that the proposed model has substantial robustness.
A systematic study of the robustness of the proposed NB random effect model will be an interesting future research topic.


\section*{Acknowledgements}
Woojoo Lee was supported by the Basic Science Research Program through the National Research Foundation of Korea (NRF) funded by the Ministry of Education (NRF-2016R1D1A1B03936100).
 Jae Youn Ahn was supported by a National Research Foundation of Korea (NRF) grant funded
by the Korean Government (NRF-2017R1D1A1B03032318).

\bibliographystyle{apalike}
\bibliography{CTE_Bib_HIX}

\begin{thebibliography}{}

\bibitem[Antonio and Valdez, 2012]{Ant12}
Antonio, K. and Valdez, E.~A. (2012).
\newblock Statistical concepts of a priori and a posteriori risk classification
  in insurance.
\newblock {\em AStA Advances in Statistical Analysis}, 96(2):187--224.

\bibitem[Boucher and Denuit, 2006]{Bou06}
Boucher, J.-P. and Denuit, M. (2006).
\newblock Fixed versus random effects in poisson regression models for claim
  counts: A case study with motor insurance.
\newblock {\em Astin Bulletin}, 36(01):285--301.

\bibitem[Boucher and Denuit, 2008]{Boucher2008credibility}
Boucher, J.-P. and Denuit, M. (2008).
\newblock Credibility premiums for the zero-inflated poisson model and new
  hunger for bonus interpretation.
\newblock {\em Insurance: Mathematics and Economics}, 42(2):727--735.

\bibitem[B{\"u}hlmann and Gisler, 2006]{Buhlmann2006}
B{\"u}hlmann, H. and Gisler, A. (2006).
\newblock {\em A course in credibility theory and its applications}.
\newblock Springer Science \& Business Media.

\bibitem[Chesher, 1984]{Che84}
Chesher, A. (1984).
\newblock Testing for neglected heterogeneity.
\newblock {\em Econometrica}, 52:865--872.

\bibitem[De~Jong and Heller, 2008]{DeJong08}
De~Jong, P. and Heller, G.~Z. (2008).
\newblock {\em Generalized linear models for insurance data}.
\newblock Cambridge University Press.

\bibitem[Denuit et~al., 2007]{Den07}
Denuit, M., Mar{\'e}chal, X., Pitrebois, S., and Walhin, J.-F. (2007).
\newblock {\em Actuarial modelling of claim counts: Risk classification,
  credibility and bonus-malus systems}.
\newblock John Wiley \& Sons.

\bibitem[Frees, 2010]{Frees10}
Frees, E.~W. (2010).
\newblock {\em Regression modeling with actuarial and financial applications}.
\newblock Cambridge University Press.

\bibitem[Lawless, 1987]{Law87}
Lawless, J.~F. (1987).
\newblock Negative binomial and mixed poisson regression.
\newblock {\em The Canadian Journal of Statistics}, 15:209--225.

\bibitem[Liang, 1987]{Lia87}
Liang, K. (1987).
\newblock A locally most powerful test for homogeneity with many strata.
\newblock {\em Biometrika}, 74:259--264.

\bibitem[Murray and Lucas, 2013]{Mur13}
Murray, Jared~S, D. D. B. C.~L. and Lucas, J.~E. (2013).
\newblock Bayesian gaussian copula factor models for mixed data.
\newblock {\em Journal of the American Statistical Association},
  108(502):656--665.

\bibitem[Pinquet, 1997]{Pin97}
Pinquet, J. (1997).
\newblock Allowance for cost of claims in bonus-malus systems.
\newblock {\em Astin Bulletin}, 27(01):33--57.

\bibitem[Pinquet, 1998]{Pin98}
Pinquet, J. (1998).
\newblock Designing optimal bonus-malus systems from different types of claims.
\newblock {\em Astin Bulletin}, 28(02):205--220.

\bibitem[Yip and Yau, 2005]{Yip2005modeling}
Yip, K.~C. and Yau, K.~K. (2005).
\newblock On modeling claim frequency data in general insurance with extra
  zeros.
\newblock {\em Insurance: Mathematics and Economics}, 36(2):153--163.

\end{thebibliography}

\appendix
\section{Auxiliary result from Theorem \ref{theo2}}

\begin{lemma}\label{appen.lem}
  Under the settings in Theorem \ref{theo2}, the analytic forms of $\widehat{I}_{\sigma^2\sigma^2}$, $\widehat{I}_{\sigma^2{\boldsymbol \omega}}$ and
$\widehat{I}_{{\boldsymbol \omega}{\boldsymbol \omega}}$ are given as follows.

\begin{equation*}
\begin{cases}
  \E{\left( \frac{\partial \ell_{i}}{\partial \sigma^{2}} \right)^{2}}
 =
\frac{1}{4} \left(
  \sum_t \frac{2\lambda_{it}^2(1+\alpha)}{(1+\alpha\lambda_{it})^2} +
  4\sum_{t<t'} \frac{\lambda_{it}(1+\alpha\lambda_{it})}{(1+\alpha\lambda_{it})^2}
  \frac{\lambda_{it'}(1+\alpha\lambda_{it'})}{(1+\alpha\lambda_{it'})^2}
  \right); \\
  \E{  \frac{\partial \ell_{i}}{\partial \sigma^{2}} \frac{\partial \ell_{i}}{\partial {\boldsymbol{\beta}}} }=0;\\
    \E{ \frac{\partial \ell_{i}}{\partial \sigma^{2}} \frac{\partial \ell_{i}}{\partial \alpha} }
  = \frac{1}{2} \sum_{t} \frac{\lambda_{it}^2}{(1+\alpha\lambda_{it})^2}; \\
    \E{ - \frac{\partial^{2} \ell_{i}}{\partial {\boldsymbol{\beta}} \partial \alpha} } = 0 ;\\
     \E{\left( \frac{\partial \ell_{i}}{\partial \alpha} \right)^{2} }=
\sum_t \left( \alpha^{-4} (
\sum_{j=0}^{\infty}
(\alpha^{-1} + j)^{-2} \Pr(N_{it} \geq j + 1) -
\frac{\alpha \lambda_{it}}{\lambda_{it} + \alpha^{-1}})
\right).
\end{cases}
\end{equation*}

\end{lemma}
\begin{proof}
For $N_{it} \sim NB(\lambda_{it}, \alpha)$, the following basic results are repeatedly used:
\begin{eqnarray*}
\E{N_{it} - \lambda_{it})^2} & = & \lambda_{it}(1 + \alpha \lambda_{it} \\
\E{N_{it} - \lambda_{it})^3} & = & \lambda_{it}(1 + \alpha \lambda_{it})(1 + 2\alpha \lambda_{it}) \\
\E{N_{it} - \lambda_{it})^4} & = & \lambda_{it}(1 + \alpha \lambda_{it})(1 + 3\lambda_{it} + 6\alpha \lambda_{it} + 3\alpha \lambda_{it}^{2} + 6\alpha^2 \lambda_{it}^2)
\end{eqnarray*}

and
\begin{eqnarray*}
\frac{\partial \ell_{i}}{\partial \sigma^{2}} & = & \frac{1}{2}
\left( (\sum_{t} \frac{n_{it}-\lambda_{it}}{1+\alpha\lambda_{it}})^2 -
\sum_{t}\frac{n_{it}(1+\alpha\lambda_{it})^2-\alpha^2 n_{it}\lambda_{it}^2 - \alpha\lambda_{it}^2}{(1+\alpha\lambda_{it})^2}\right) \\
 & = & \frac{1}{2} \left( (\sum_{t} \frac{n_{it}-\lambda_{it}}{1+\alpha\lambda_{it}} )^2 -
 \sum_{t}\frac{(1+2\alpha\lambda_{it})(n_{it} - \lambda_{it}) + \lambda_{it}(1+\alpha\lambda_{it})}{(1+\alpha\lambda_{it})^2} \right)
\end{eqnarray*}

From the first equality, we have
\begin{equation*}
\begin{aligned}
  \E{ \left( \frac{\partial \ell_{i}}{\partial \sigma^{2}} \right)^{2}} &= \frac{1}{4} \E{\left(\sum_{t} \frac{N_{it}-\lambda_{it}}{1+\alpha\lambda_{it}}\right)^4} +
  \frac{1}{4} \E{\left( \sum_{t}\frac{(1+2\alpha\lambda_{it})(N_{it} - \lambda_{it}) + \lambda_{it}(1+\alpha\lambda_{it})}{(1+\alpha\lambda_{it})^2} \right)^2} \\
  &\quad\quad\quad\quad\quad\quad -  \frac{1}{2} \E{ \left( \sum_{t} \frac{N_{it}-\lambda_{it}}{1+\alpha\lambda_{it}} \right)^2
  \sum_{t}\frac{(1+2\alpha\lambda_{it})(N_{it} - \lambda_{it}) + \lambda_{it}(1+\alpha\lambda_{it})}{(1+\alpha\lambda_{it})^2}  } \\
 & =  \frac{1}{4} \left( \sum_t \E{\left( \frac{N_{it} - \lambda_{it}}{1 + \alpha \lambda_{it}} \right)^4} +
  6 \sum_{t<t'} \E{\left( \frac{N_{it} - \lambda_{it}}{1 + \alpha \lambda_{it}} \right)^2
  ( \frac{N_{it'} - \lambda_{it'}}{1 + \alpha \lambda_{it'}} )^2 \right)} \\
 &\quad\quad +  \frac{1}{4} \bigg( \sum_{t} \frac{(1+2\alpha\lambda_{it})^2 E(N_{it}-\lambda_{it})^2 + \lambda_{it}^2(1+\alpha \lambda_{it})^2}{(1+\alpha\lambda_{it})^4} \\
 &\quad\quad\quad\quad\quad\quad\quad\quad\quad\quad\quad\quad\quad\quad\quad\quad\quad\quad + 2 \sum_{t<t'} \frac{\lambda_{it}(1+\alpha\lambda_{it})}{(1+\alpha\lambda_{it})^2} \frac{\lambda_{it'}(1+\alpha\lambda_{it'})}{(1+\alpha\lambda_{it'})^2}
\bigg) \\
 &\quad\quad -  \frac{1}{2} \bigg( \sum_{t} \frac{(1+2\alpha\lambda_{it}) E(N_{it}-\lambda_{it})^3 + \lambda_{it}(1+\alpha \lambda_{it})E(N_{it}-\lambda_{it})^2}{(1+\alpha\lambda_{it})^4} \\
 &\quad\quad\quad\quad\quad\quad\quad\quad\quad\quad\quad\quad\quad\quad\quad\quad\quad\quad + 2 \sum_{t<t'} \frac{\lambda_{it}(1+\alpha\lambda_{it})}{(1+\alpha\lambda_{it})^2} \frac{\lambda_{it'}(1+\alpha\lambda_{it'})}{(1+\alpha\lambda_{it'})^2}
\bigg) \\
  & =  \frac{1}{4} \left(
  \sum_t \frac{2\lambda_{it}^2(1+\alpha)}{(1+\alpha\lambda_{it})^2} +
  4\sum_{t<t'} \frac{\lambda_{it}(1+\alpha\lambda_{it})}{(1+\alpha\lambda_{it})^2}
  \frac{\lambda_{it'}(1+\alpha\lambda_{it'})}{(1+\alpha\lambda_{it'})^2}
  \right). \\
  \end{aligned}
\end{equation*}
The second and third equalities are from
\begin{equation*}
\begin{aligned}
  &\E{ \frac{\partial \ell_{i}}{\partial \sigma^{2}} \frac{\partial \ell_{i}}{\partial {\boldsymbol{\beta}}} }\\
  & \quad=
  \E{ \frac{1}{2} \left( \left(\sum_{t} \frac{N_{it}-\lambda_{it}}{1+\alpha\lambda_{it}} \right)^2 -
 \sum_{t}\frac{(1+2\alpha\lambda_{it})(N_{it} - \lambda_{it}) + \lambda_{it}(1+\alpha\lambda_{it})}{(1+\alpha\lambda_{it})^2} \right)
  \left( \sum_{t} \frac{N_{it} - \lambda_{it}}{1 + \alpha \lambda_{it}}  {\boldsymbol x}_{it}^{T} \right) } \\
 & \quad =  \frac{1}{2} \sum_{t} \frac{\E{(N_{it}-\lambda_{it})^3} -
 (1+2\alpha\lambda_{it})\E{(N_{it}-\lambda_{it})^2} - \lambda_{it}(1+\alpha\lambda_{it})\E{N_{it}-\lambda_{it}}}{(1+\alpha\lambda_{it})^3} {\boldsymbol x}_{it}^{T}\\
 & \quad =0, \\
 \end{aligned}
\end{equation*}
 and
\begin{equation*}
\begin{aligned}
  &\E { \frac{\partial \ell_{i}}{\partial \sigma^{2}} \frac{\partial \ell_{i}}{\partial \alpha} }
   =
  \E { - \frac{\partial^{2} \ell_{i}}{\partial \sigma^{2} \partial \alpha} } \\
  &\quad\quad\quad\quad =
 -\frac{1}{2}  \E { \frac{\partial}{\partial \alpha}
    (\sum_{t} \frac{N_{it}-\lambda_{it}}{1+\alpha\lambda_{it}} )^2 }
 + \frac{1}{2}  \E { \frac{\partial}{\partial \alpha}
 \sum_{t}\frac{(1+2\alpha\lambda_{it})(N_{it} - \lambda_{it}) + \lambda_{it}(1+\alpha\lambda_{it})}{(1+\alpha\lambda_{it})^2} } \\
  &\quad\quad\quad\quad =  - \frac{1}{2} \left(-2 \sum_{t} \frac{\lambda_{it} \E{(N_{it}-\lambda_{it})^2}}{(1+\alpha\lambda_{it})^3} \right)
  + \frac{1}{2} \left( - \sum_{t} \frac{\lambda_{it}^2}{(1+\alpha\lambda_{it})^2} \right)\\
  & \quad\quad\quad\quad = \frac{1}{2} \sum_{t} \frac{\lambda_{it}^2}{(1+\alpha\lambda_{it})^2}, \\
\end{aligned}
\end{equation*}
respectively. The fourth equation is from
\begin{equation*}
\begin{aligned}
  \E { \frac{\partial \ell_{i}}{\partial {\boldsymbol{\beta}}} \frac{\partial \ell_{i}}{\partial \alpha} }
  & =
  \E { - \frac{\partial^{2} \ell_{i}}{\partial {\boldsymbol{\beta}} \partial \alpha} }\\
  & =
  \E { \frac{\partial}{\partial \alpha}
    \left\{ - \frac{N_{it} - \lambda_{it}}{1 + \alpha \lambda_{it}}  {\boldsymbol x}_{it}^{T} \right\} }\\
    & =
  - \frac{\E{ N_{it} - \lambda_{it}} \lambda_{it}}{(1 + \alpha \lambda_{it})^{2}}  {\boldsymbol x}_{it}^{T} \\
  & = 0.
\end{aligned}
\end{equation*}
Finally, following \cite{Law87}, we have the last equality as follows
\begin{equation*}
\begin{aligned}
\E{ \left( \frac{\partial \ell_{i}}{\partial \alpha} \right)^{2}} &=
\E { - \frac{\partial^{2} \ell_{i}}{\partial \alpha^{2}} }\\
& =
\sum_t \left( \alpha^{-4} \left(
\E { \sum_{j=0}^{N_{it}-1}(\alpha^{-1} + j)^{-2} } - \frac{\alpha \lambda_{it}}{\lambda_{it} + \alpha^{-1}} \right) \right) \\
&=\sum_t \left( \alpha^{-4} \left(
\sum_{j=0}^{\infty}
(\alpha^{-1} + j)^{-2} \Pr(N_{it} \geq j + 1) -
\frac{\alpha \lambda_{it}}{\lambda_{it} + \alpha^{-1}}\right)
\right).
\end{aligned}
\end{equation*}
\end{proof}

\section{Tables for the Simulation Results}
  \begin{figure}[h!]
     \centering
    \subfloat[Scenario 1 ($\sigma^2=0$)]{%
      \includegraphics[width=0.49\textwidth]{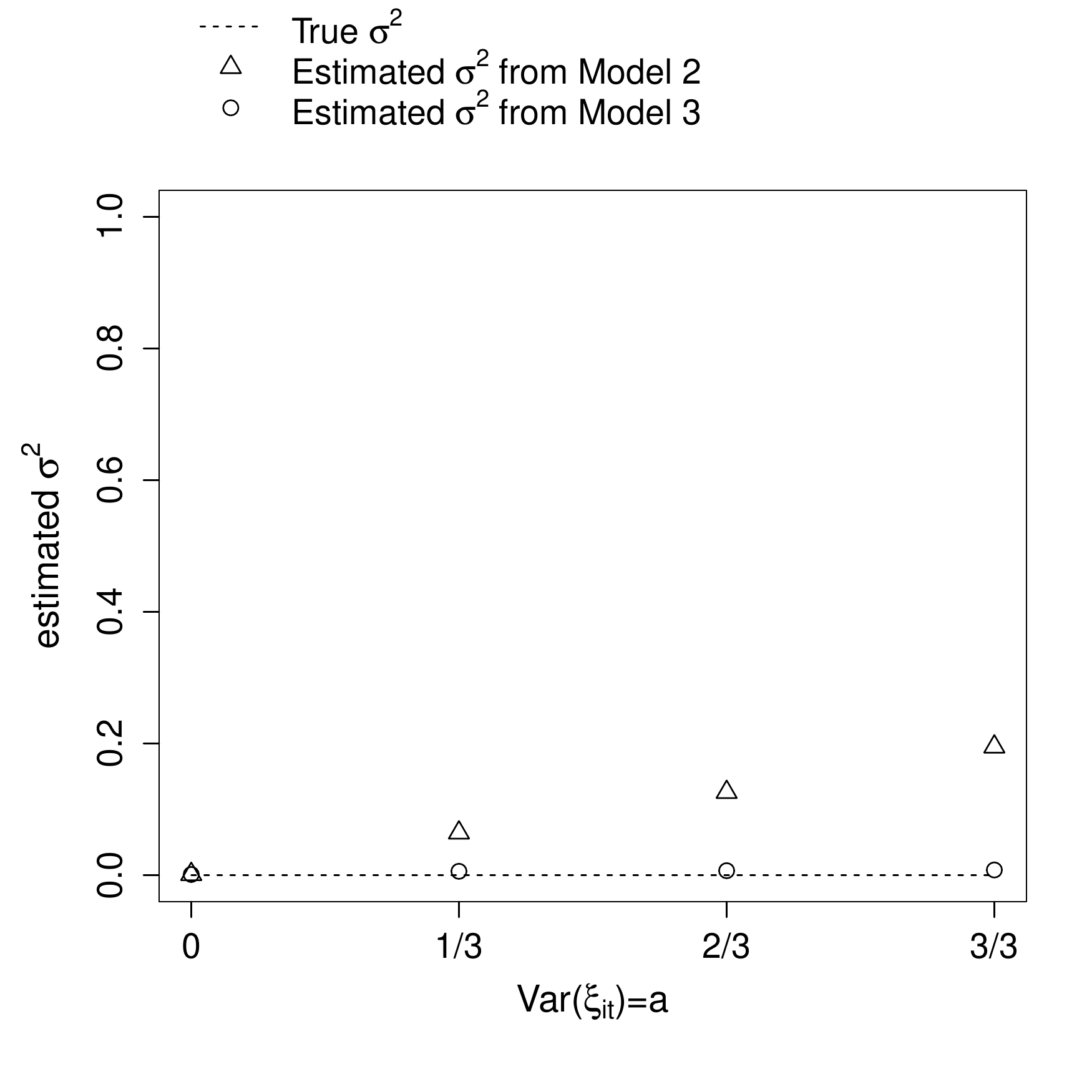}
    }
    \hfill
    \subfloat[Scenario 2 ($\sigma^2=1/6$)]{%
      \includegraphics[width=0.49\textwidth]{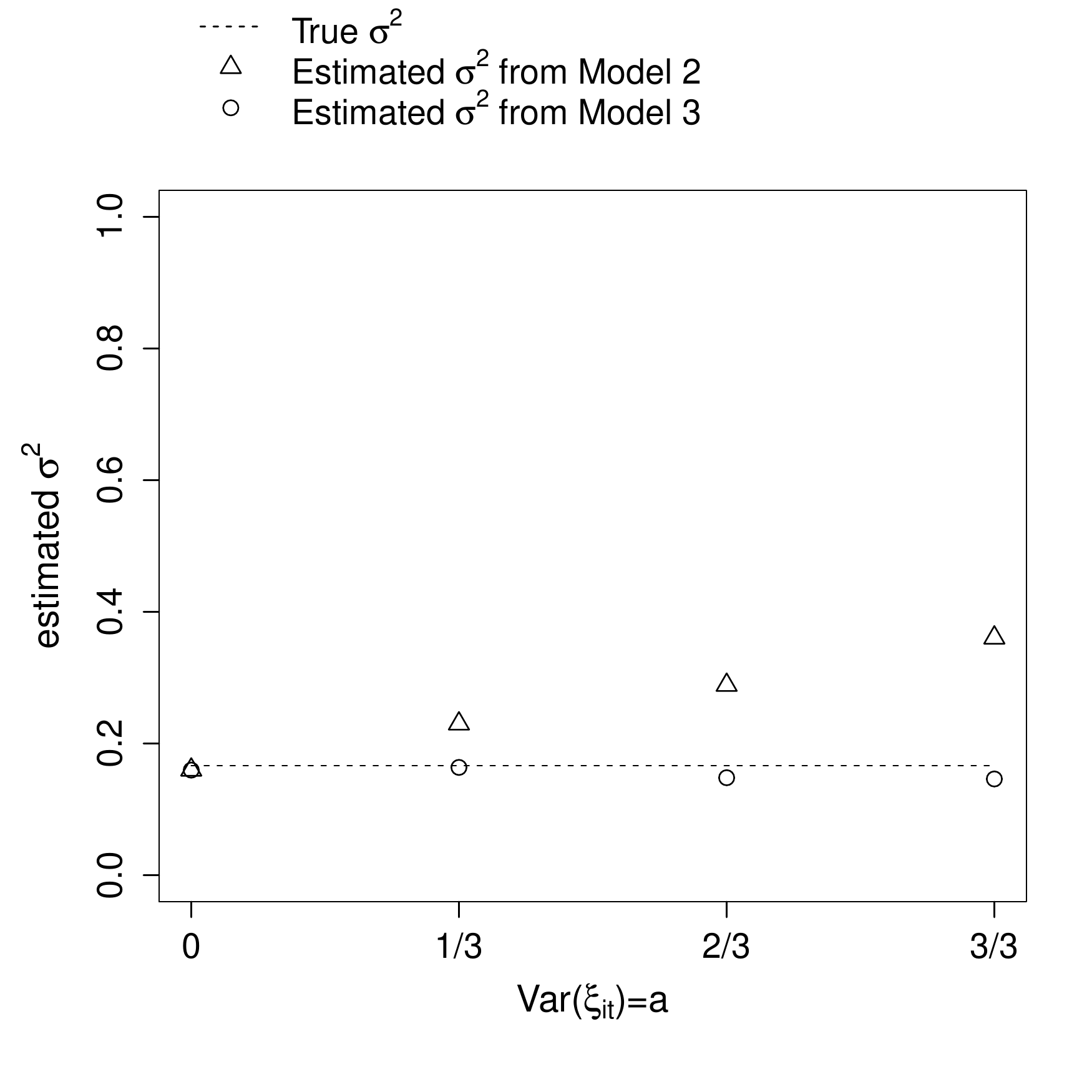}
    }

        \subfloat[Scenario 3 ($\sigma^2=2/6$)]{%
      \includegraphics[width=0.49\textwidth]{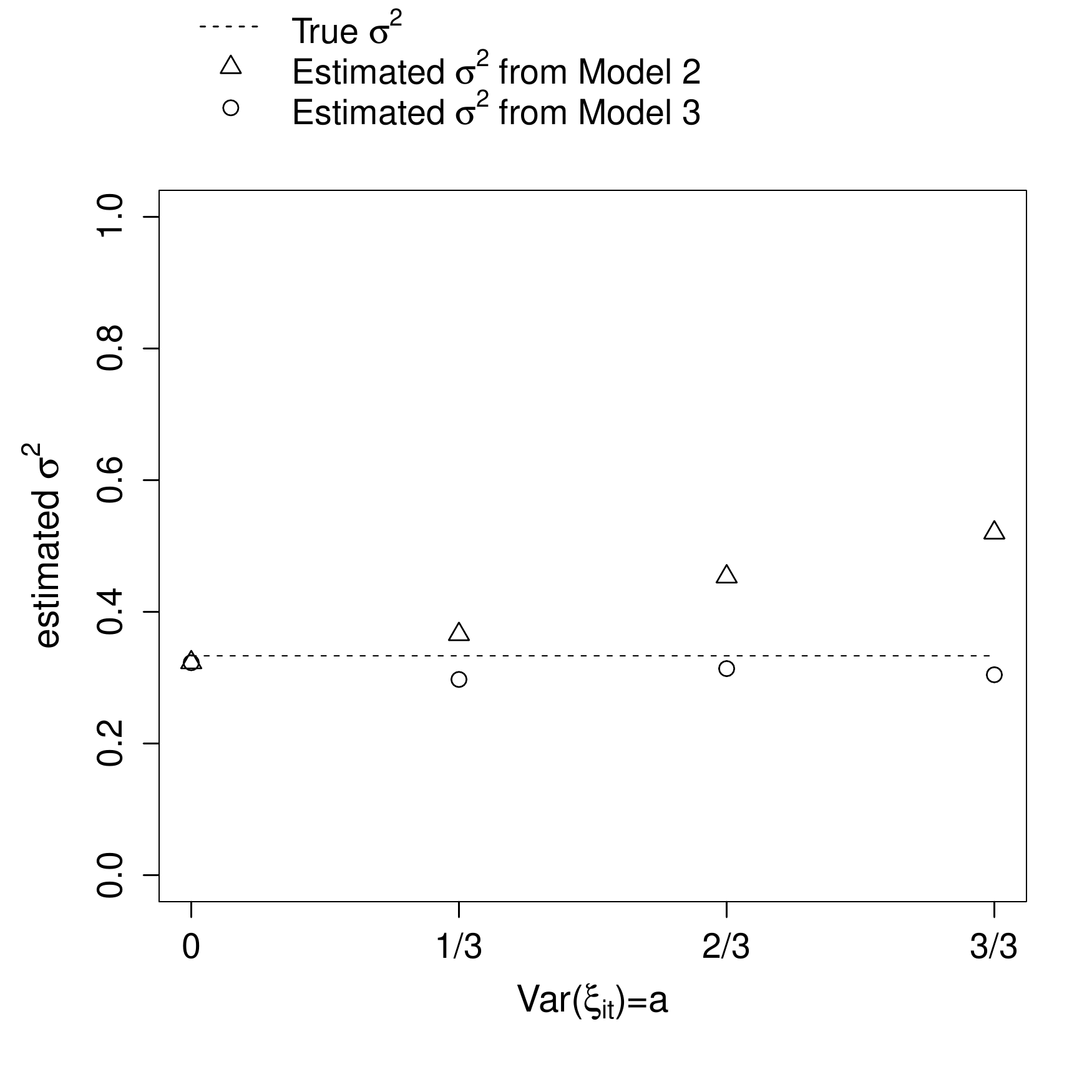}
    }
    \hfill
    \subfloat[Scenario 4 ($\sigma^2=3/6$)]{%
      \includegraphics[width=0.49\textwidth]{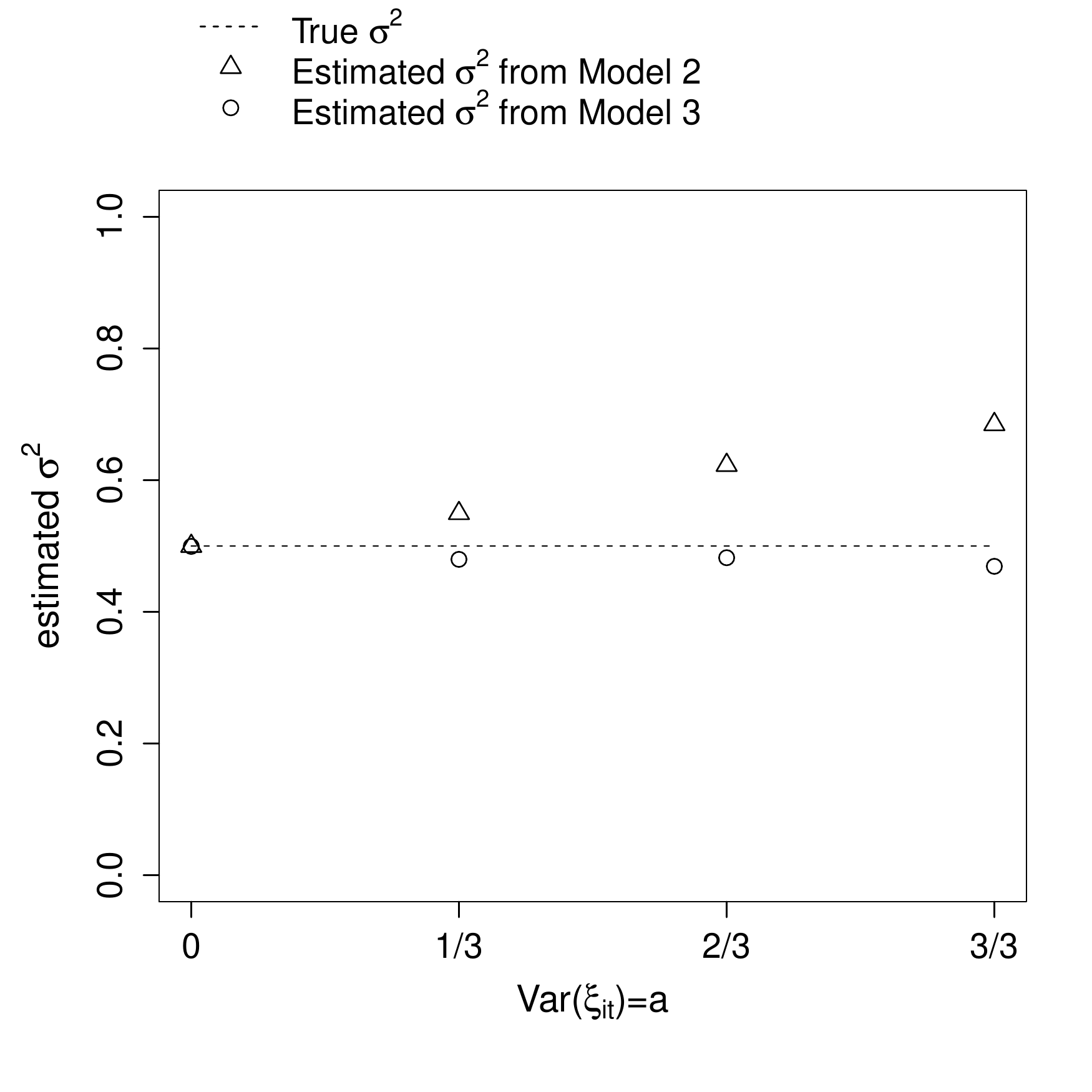}
    }
      \caption{Estimation of shared random effects when observations are from Model \ref{mod.1}.}
     \label{fig.2}
  \end{figure}

  \begin{figure}[h!]
     \centering
    \subfloat[Scenario 1 ($\sigma^2=0$)]{%
      \includegraphics[width=0.49\textwidth]{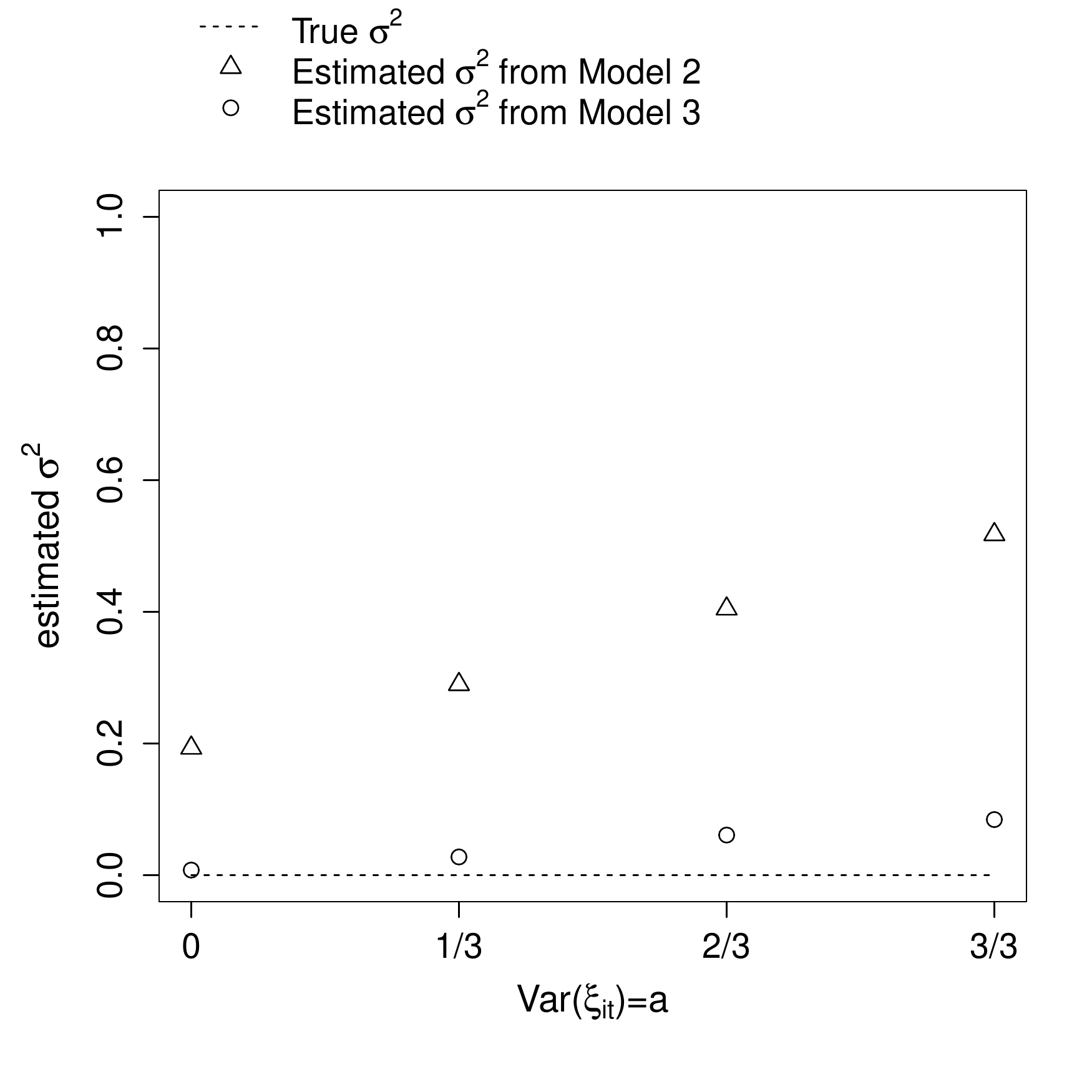}
    }
    \hfill
    \subfloat[Scenario 2 ($\sigma^2=1/6$)]{%
      \includegraphics[width=0.49\textwidth]{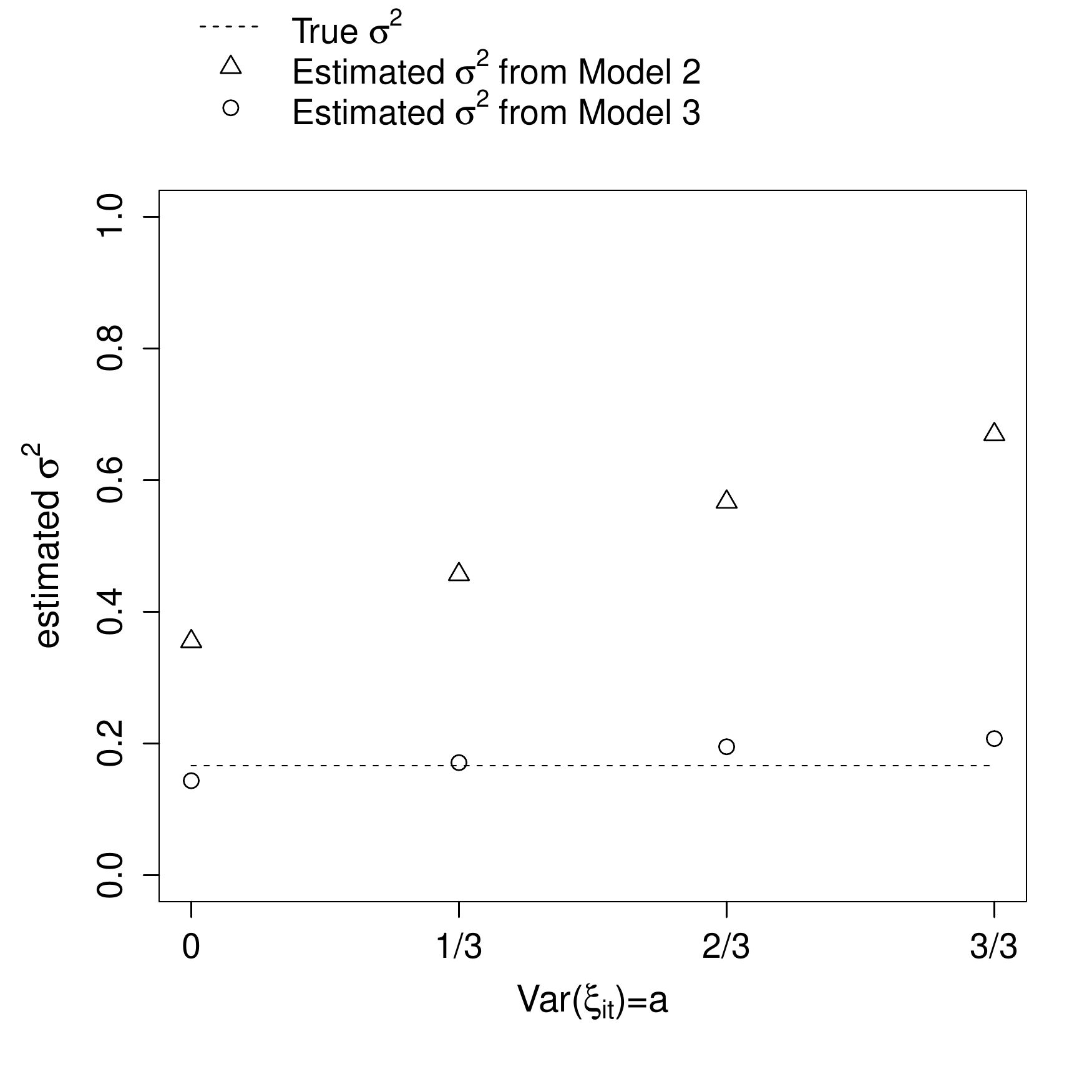}
    }

        \subfloat[Scenario 3 ($\sigma^2=2/6$)]{%
      \includegraphics[width=0.49\textwidth]{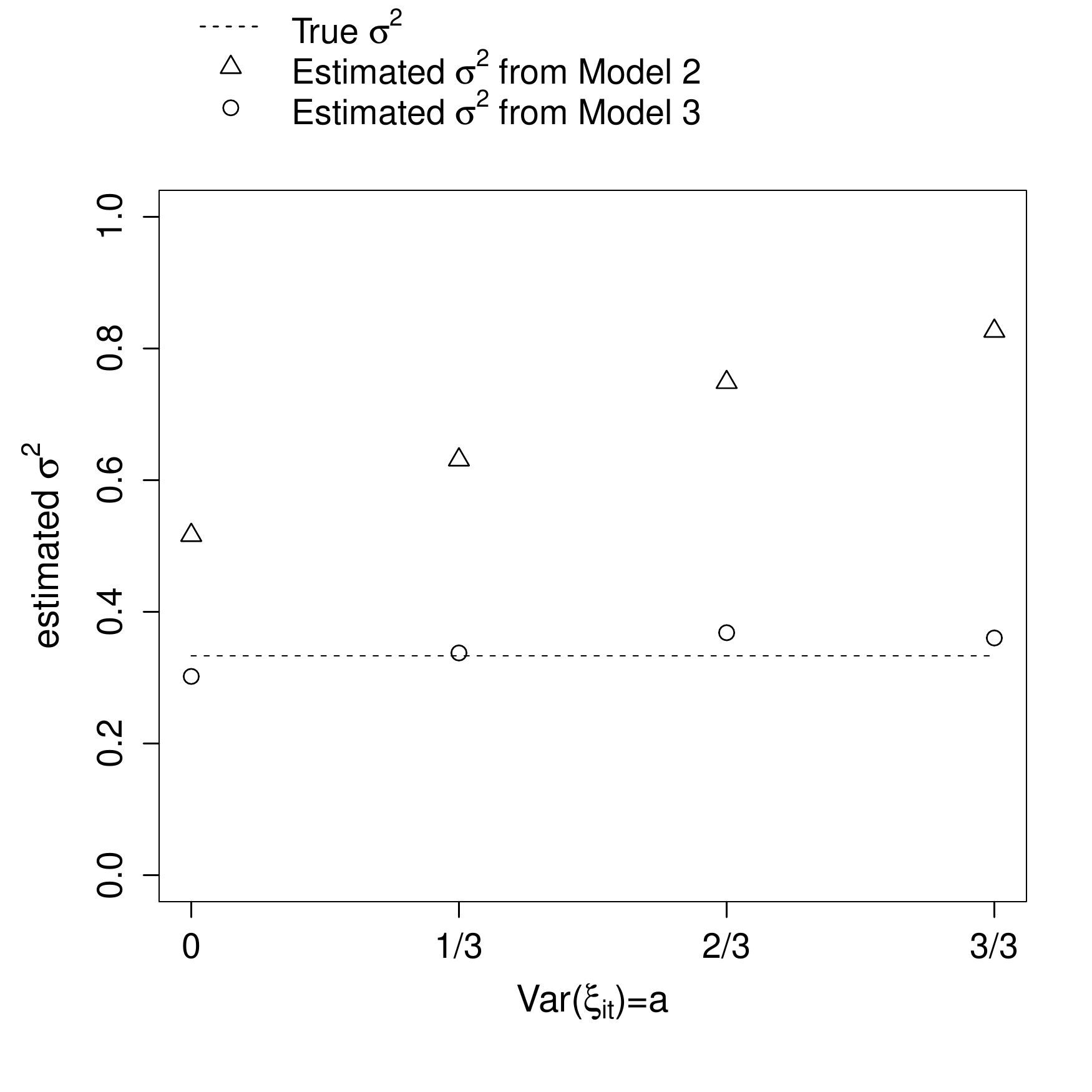}
    }
    \hfill
    \subfloat[Scenario 4 ($\sigma^2=3/6$)]{%
      \includegraphics[width=0.49\textwidth]{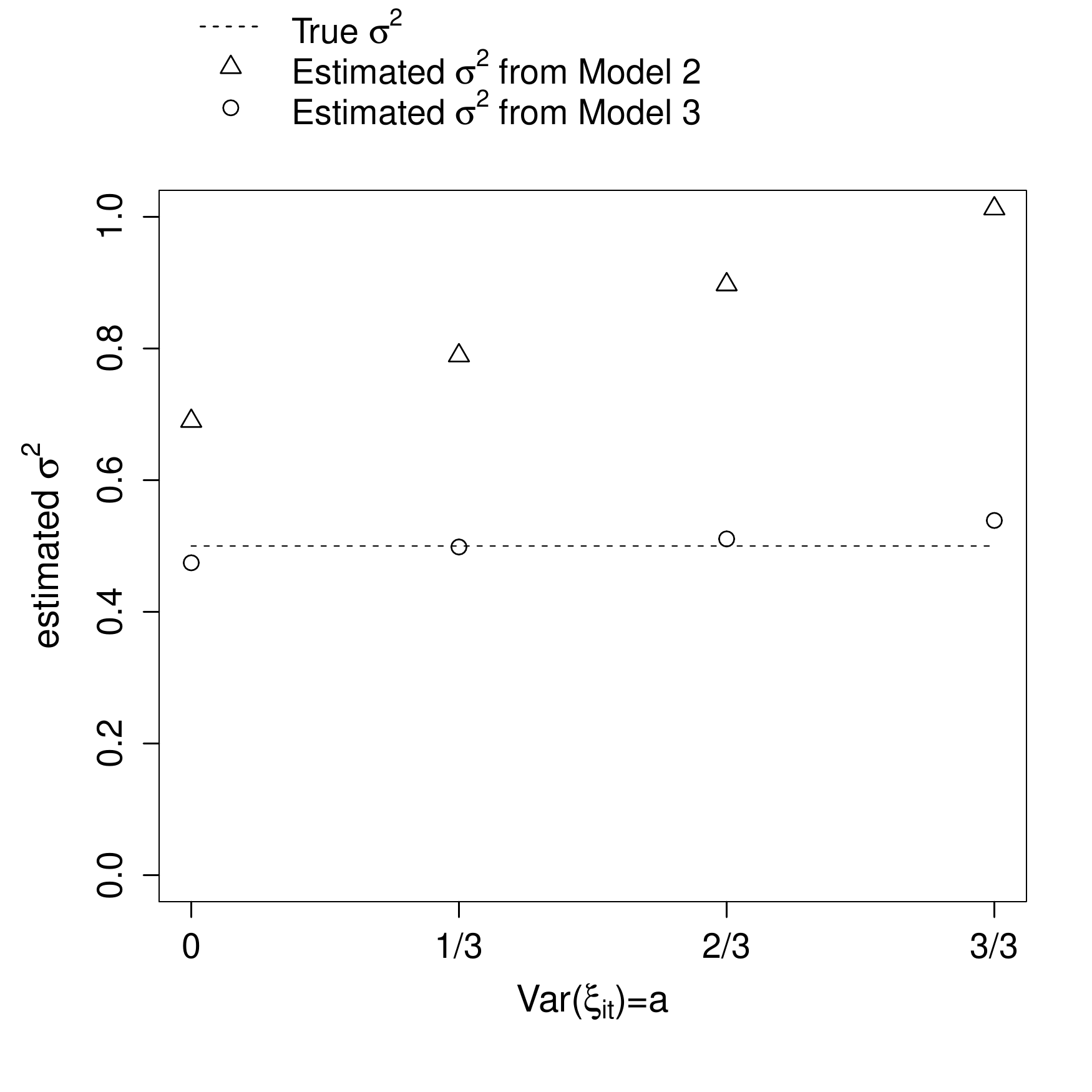}
    }
      \caption{Estimation of shared random effects when observations are from Alternative Model \ref{mod.4}.}
     \label{fig.3}
  \end{figure}

  \begin{figure}[h!]
     \centering
    \subfloat[Scenario 1 ($\sigma^2=0$)]{%
      \includegraphics[width=0.49\textwidth]{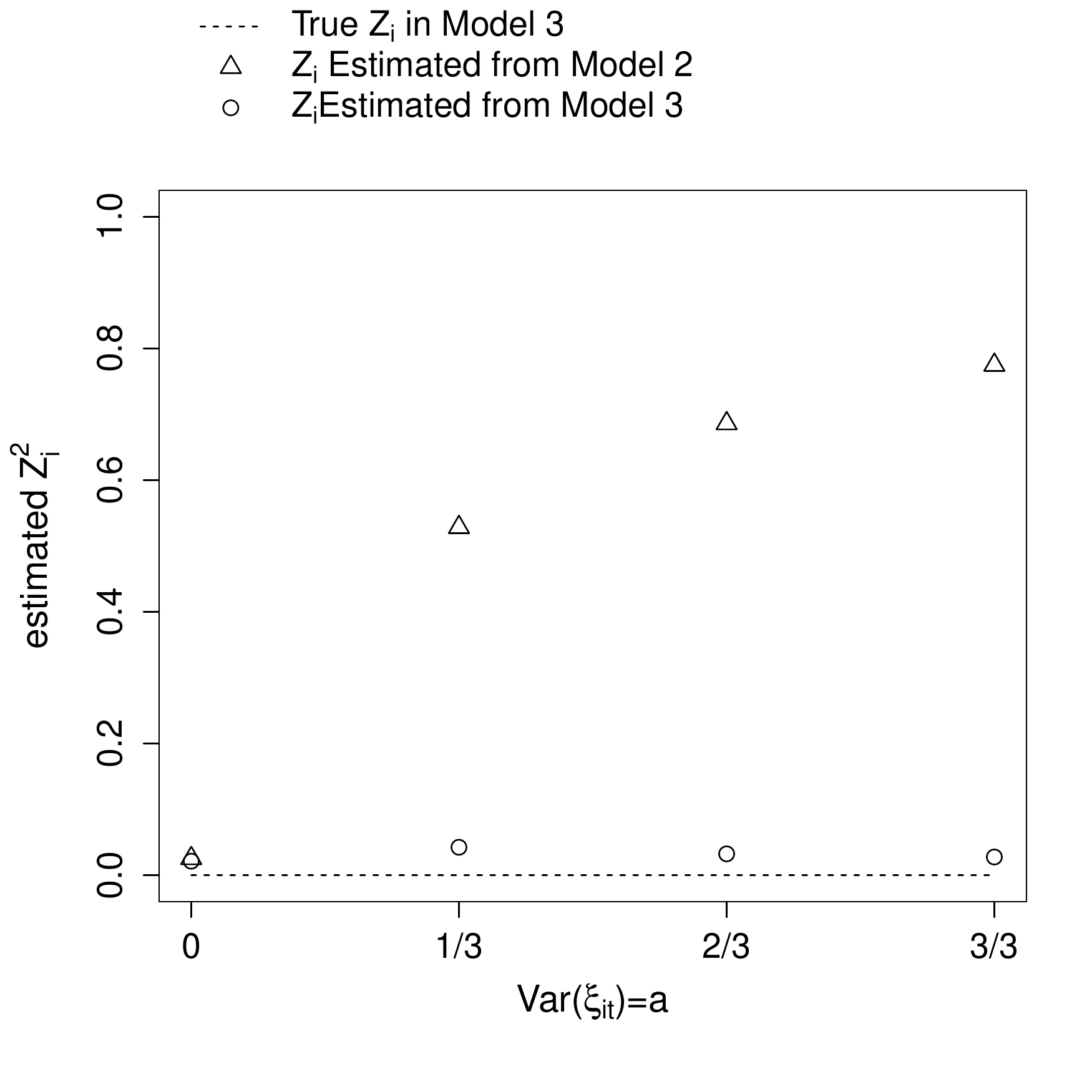}
    }
    \hfill
    \subfloat[Scenario 2 ($\sigma^2=1/6$)]{%
      \includegraphics[width=0.49\textwidth]{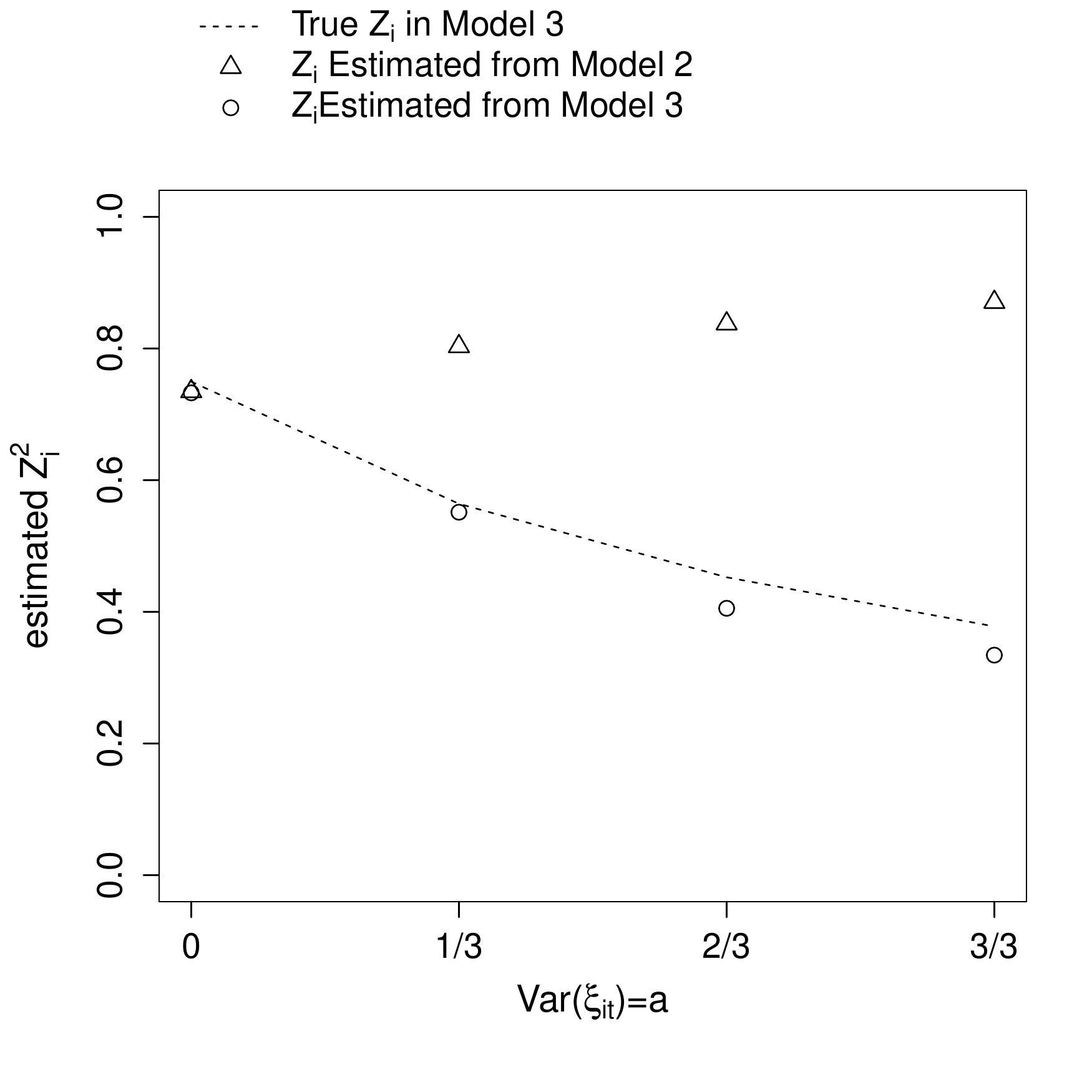}
    }

        \subfloat[Scenario 3 ($\sigma^2=2/6$)]{%
      \includegraphics[width=0.49\textwidth]{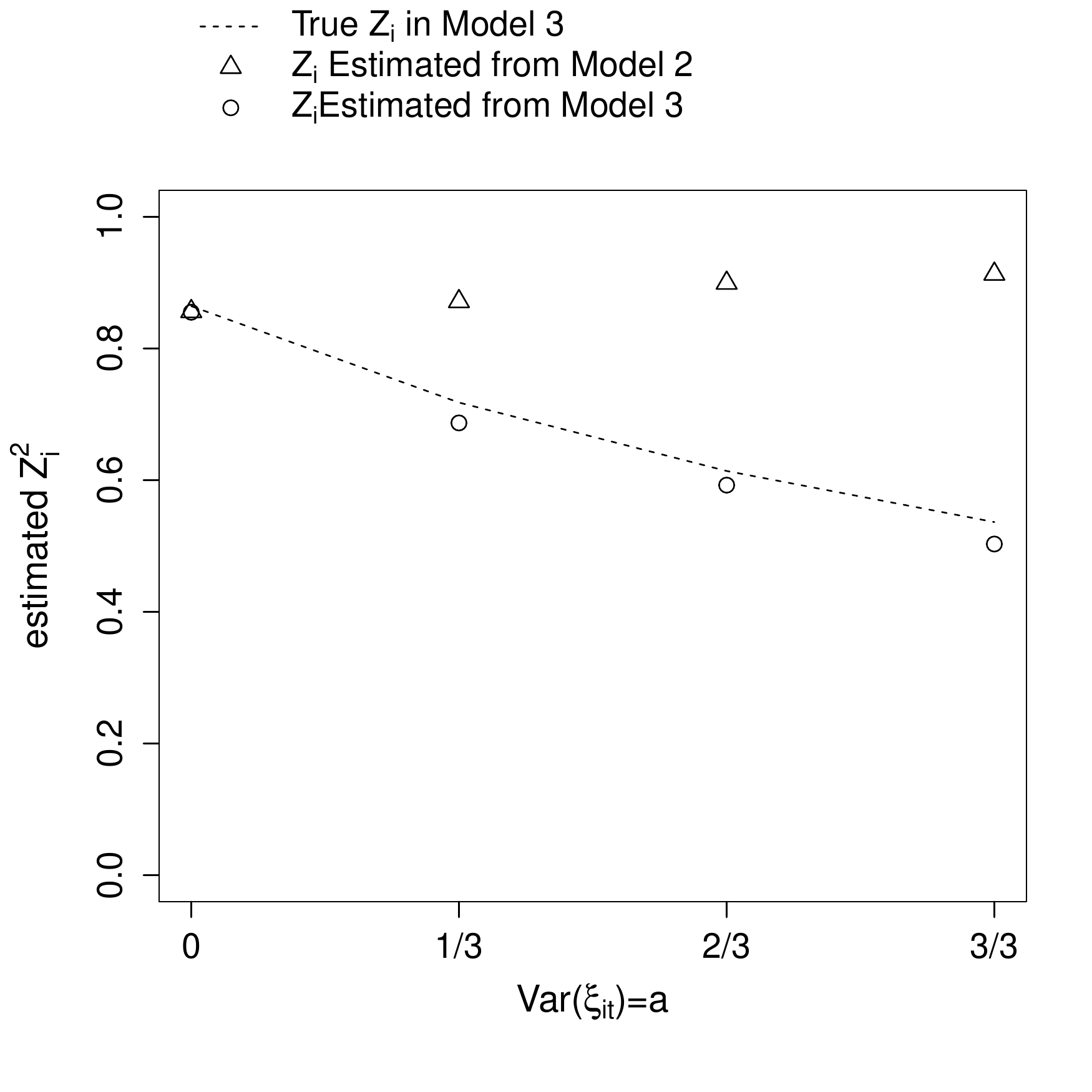}
    }
    \hfill
    \subfloat[Scenario 4 ($\sigma^2=3/6$)]{%
      \includegraphics[width=0.49\textwidth]{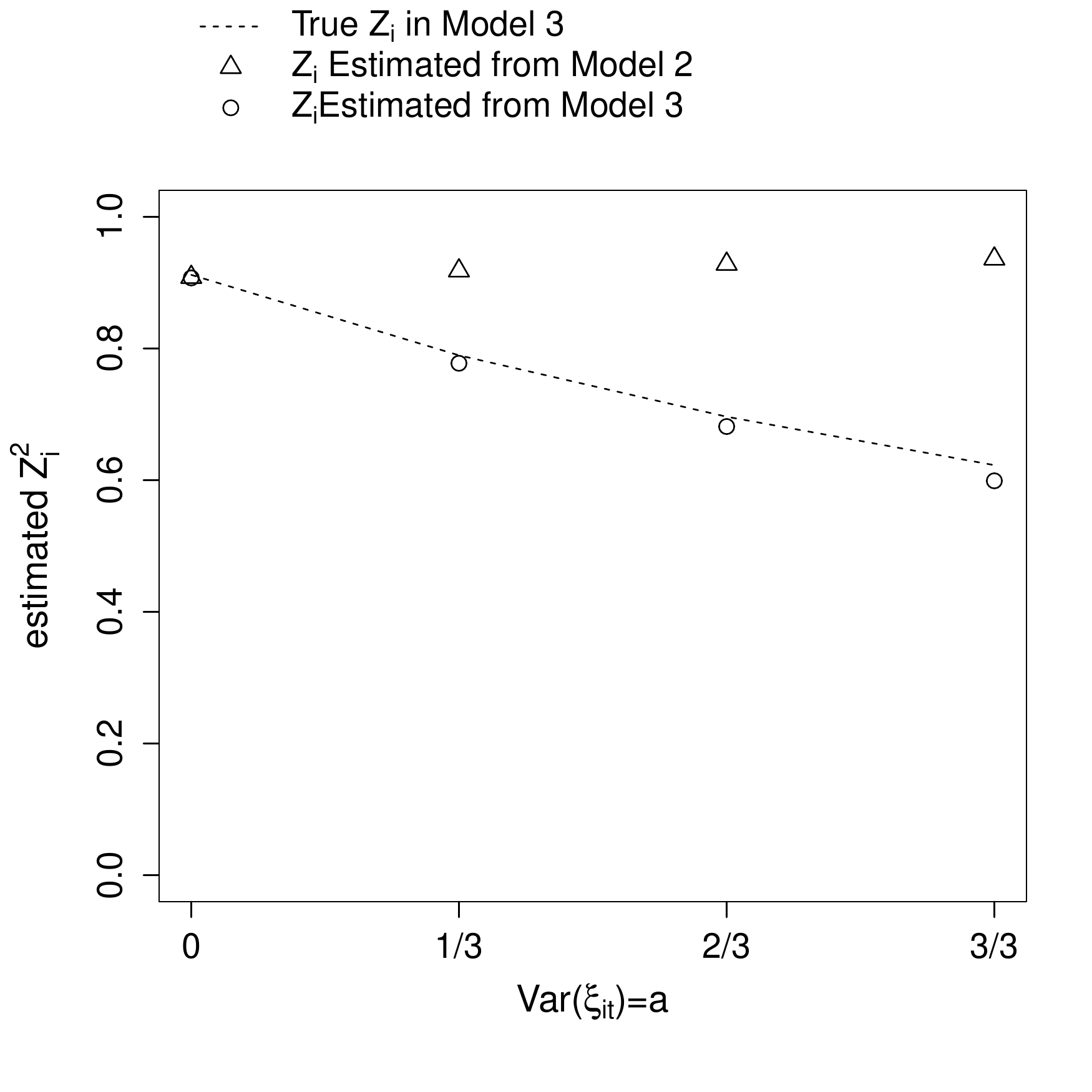}
    }
      \caption{Estimation of B\"uhlmann factor when observations are from Model \ref{mod.1}.}
     \label{fig.4}
  \end{figure}

  \begin{figure}[h!]
     \centering
    \subfloat[Scenario 1 ($\sigma^2=0$)]{%
      \includegraphics[width=0.49\textwidth]{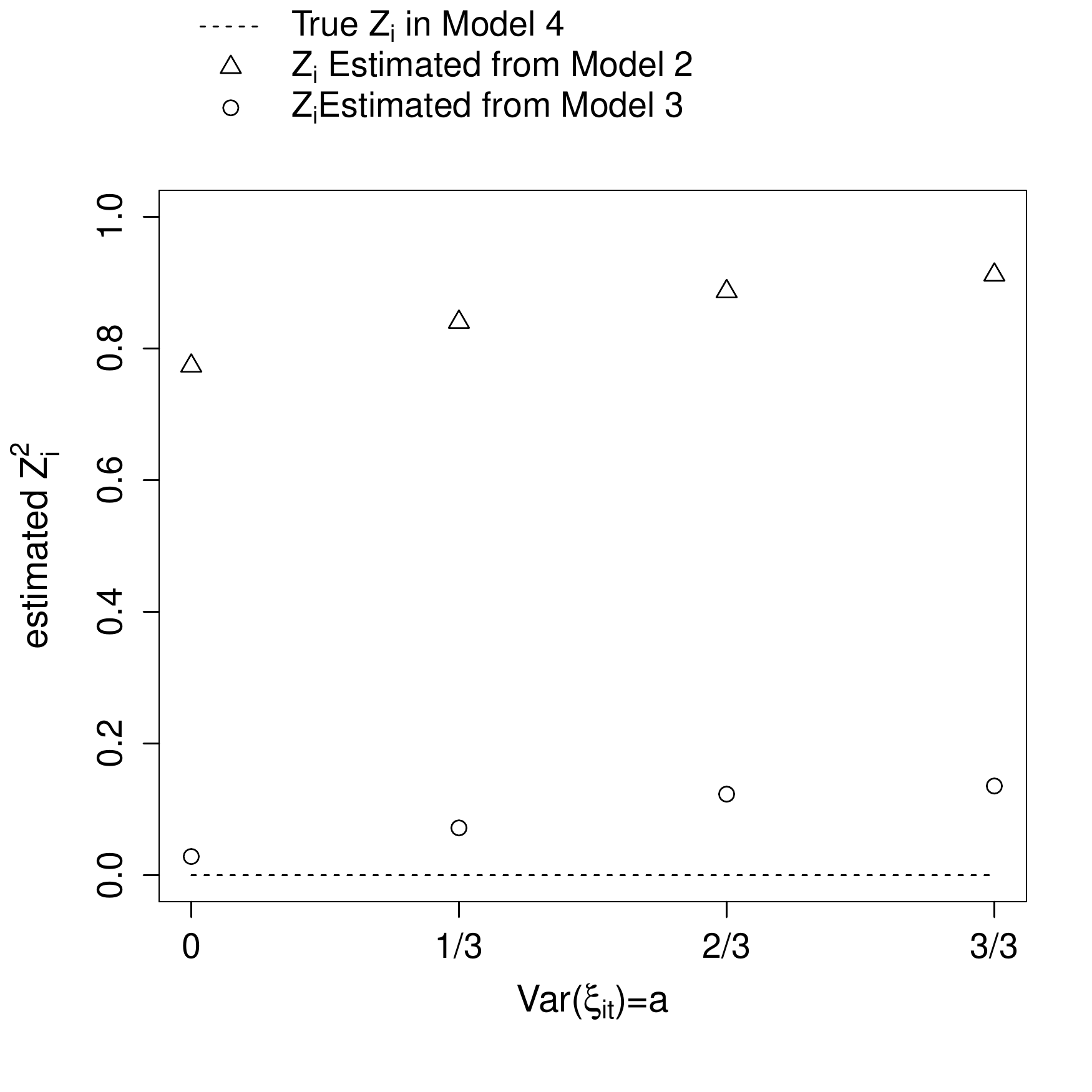}
    }
    \hfill
    \subfloat[Scenario 2 ($\sigma^2=1/6$)]{%
      \includegraphics[width=0.49\textwidth]{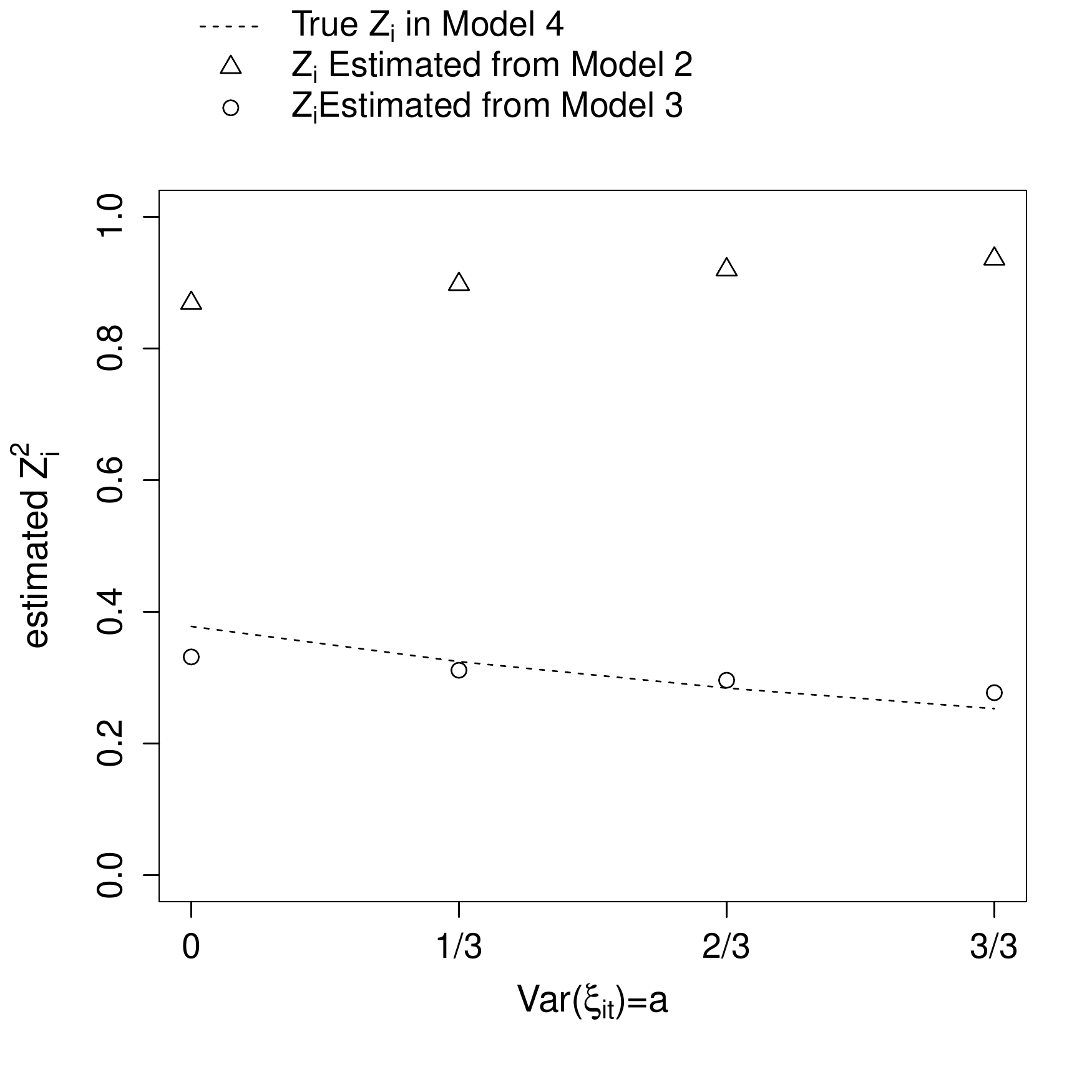}
    }

        \subfloat[Scenario 3 ($\sigma^2=2/6$)]{%
      \includegraphics[width=0.49\textwidth]{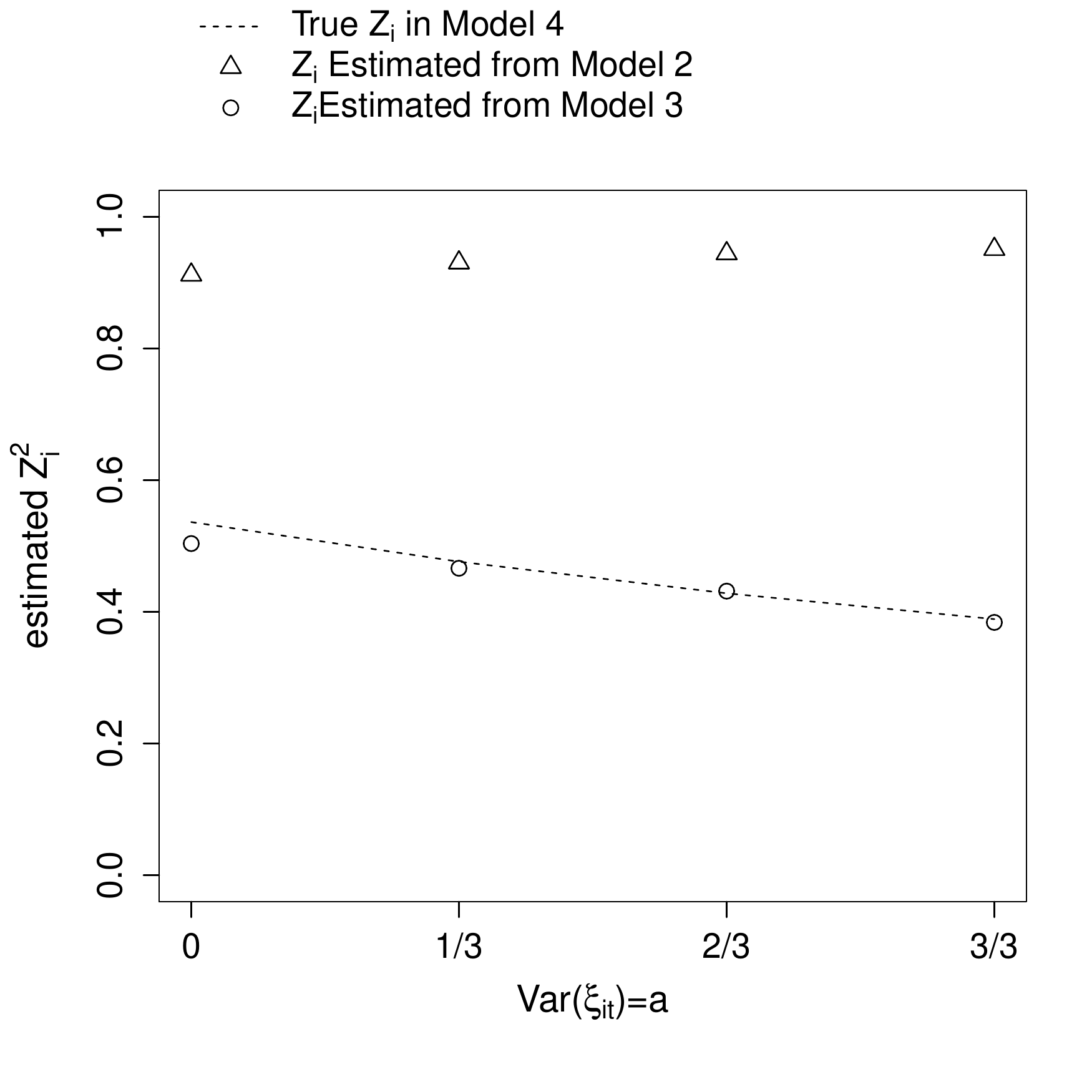}
    }
    \hfill
    \subfloat[Scenario 4 ($\sigma^2=3/6$)]{%
      \includegraphics[width=0.49\textwidth]{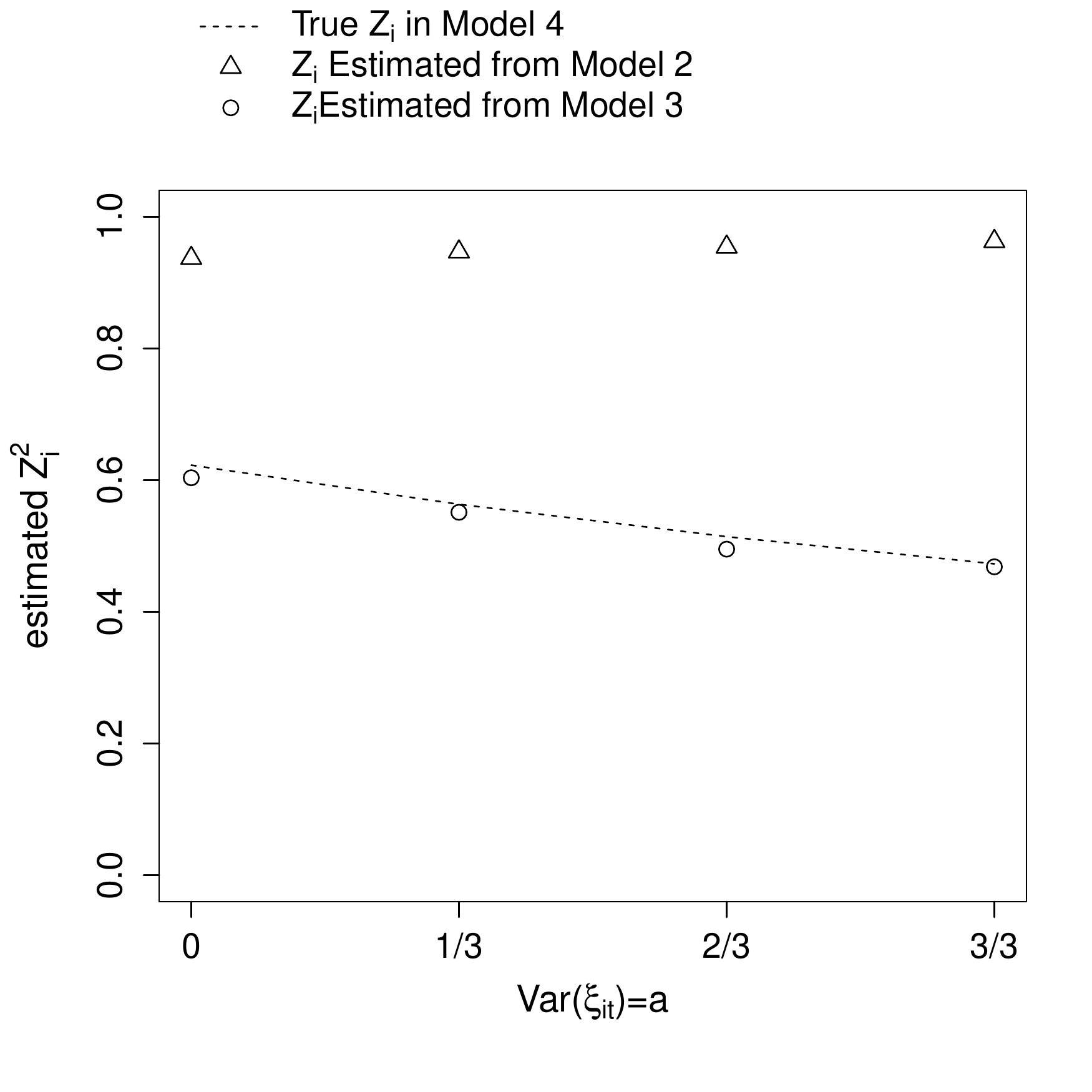}
    }
      \caption{Estimation of B\"uhlmann factor when observations are from Alternative Model \ref{mod.4}.}
     \label{fig.5}
  \end{figure}

  \begin{figure}[h!]
     \centering
    \subfloat[Scenario 1 ($\sigma^2=0$)]{%
      \includegraphics[width=0.49\textwidth]{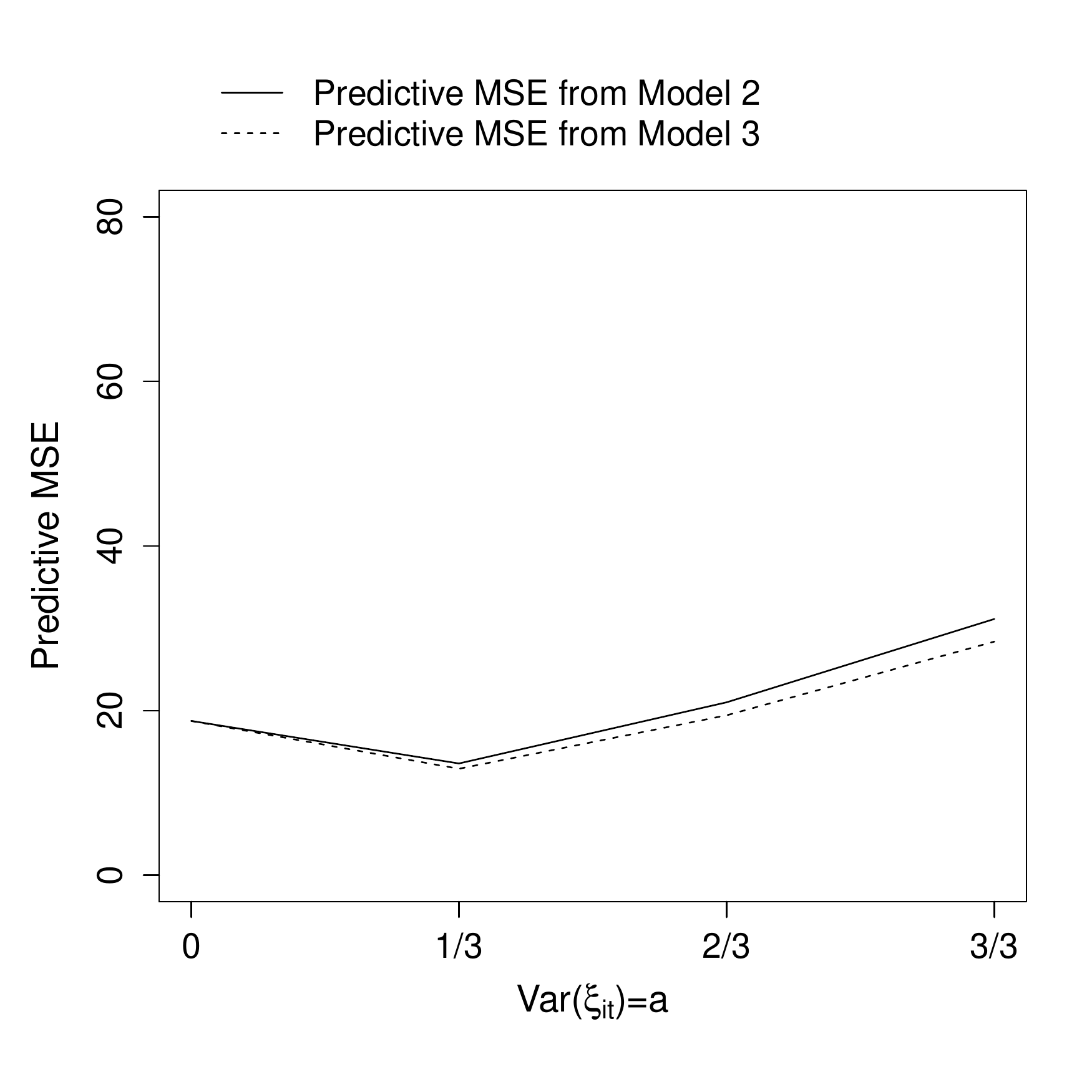}
    }
    \hfill
    \subfloat[Scenario 2 ($\sigma^2=1/6$)]{%
      \includegraphics[width=0.49\textwidth]{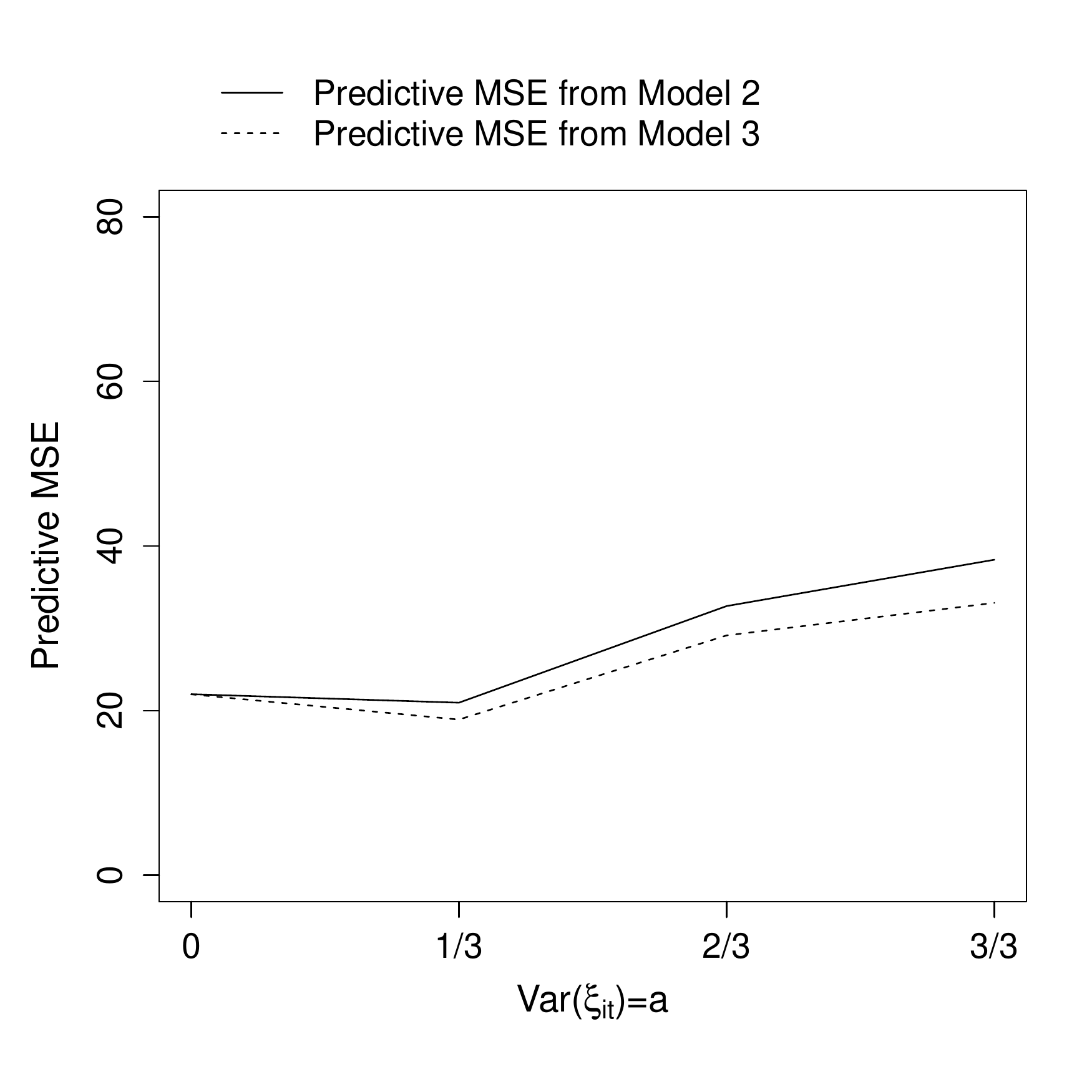}
    }

        \subfloat[Scenario 3 ($\sigma^2=2/6$)]{%
      \includegraphics[width=0.49\textwidth]{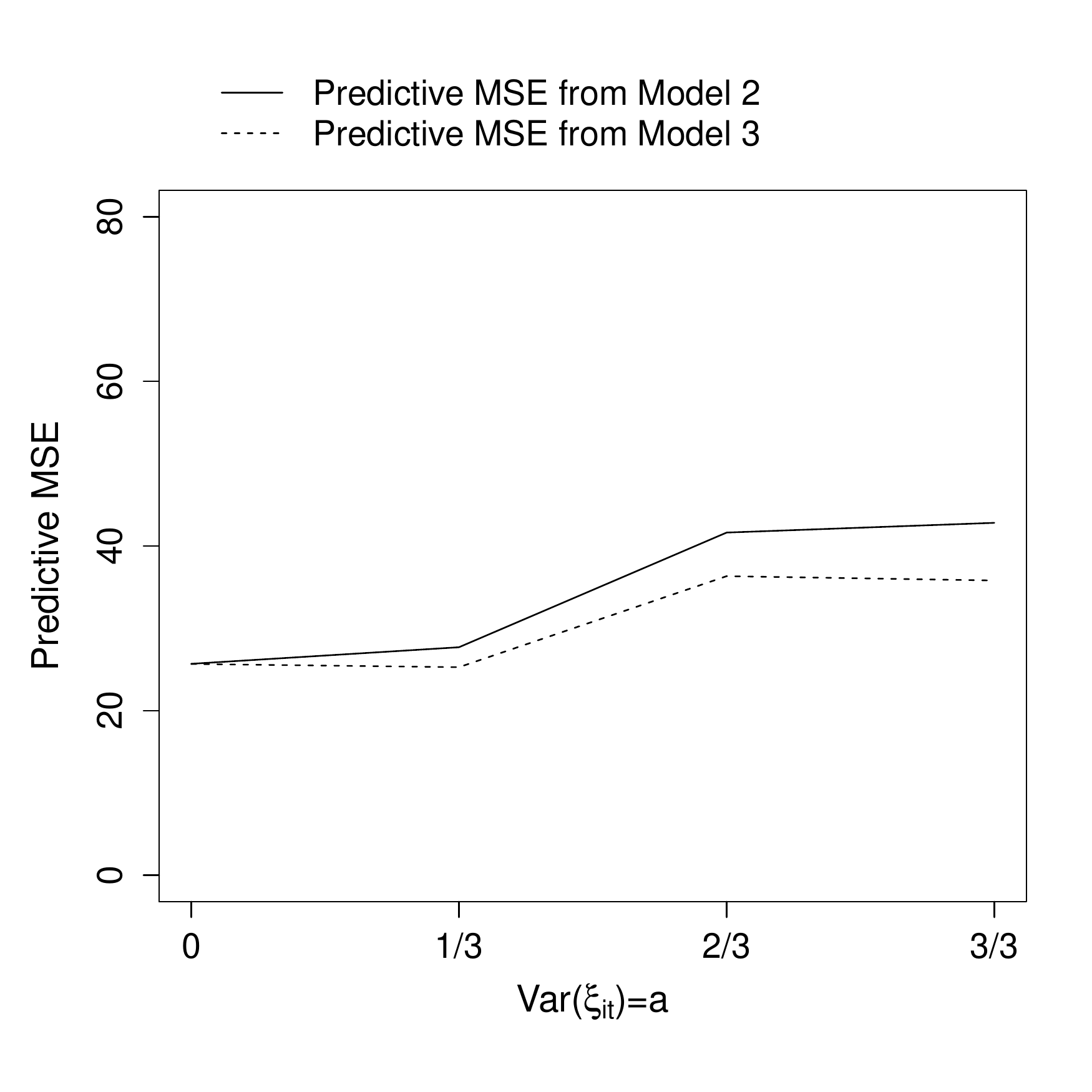}
    }
    \hfill
    \subfloat[Scenario 4 ($\sigma^2=3/6$)]{%
      \includegraphics[width=0.49\textwidth]{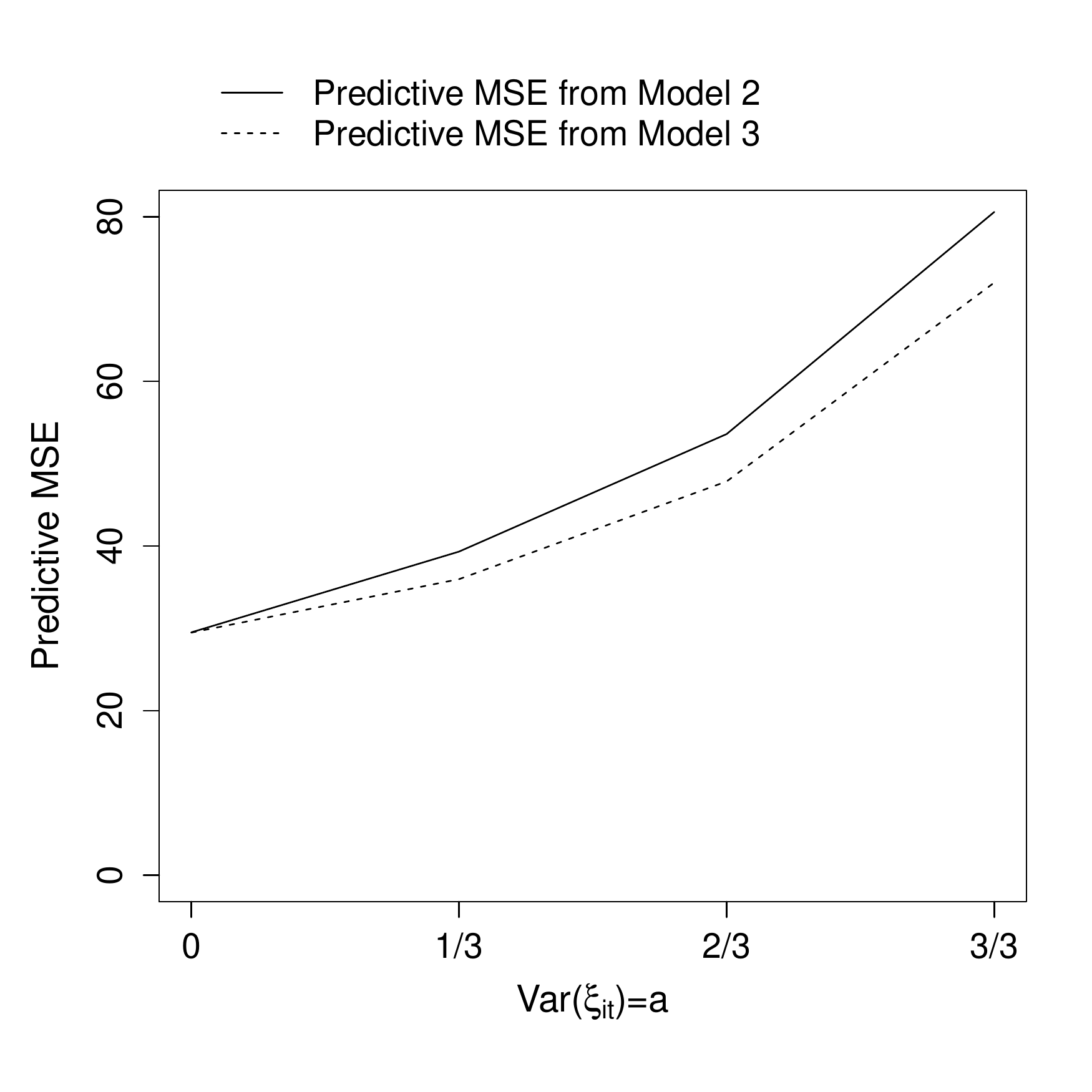}
    }
      \caption{Estimation of Predictive MSE when observations are from Model \ref{mod.1}.}
     \label{fig.6}
  \end{figure}

  \begin{figure}[h!]
     \centering
    \subfloat[Scenario 1 ($\sigma^2=0$)]{%
      \includegraphics[width=0.49\textwidth]{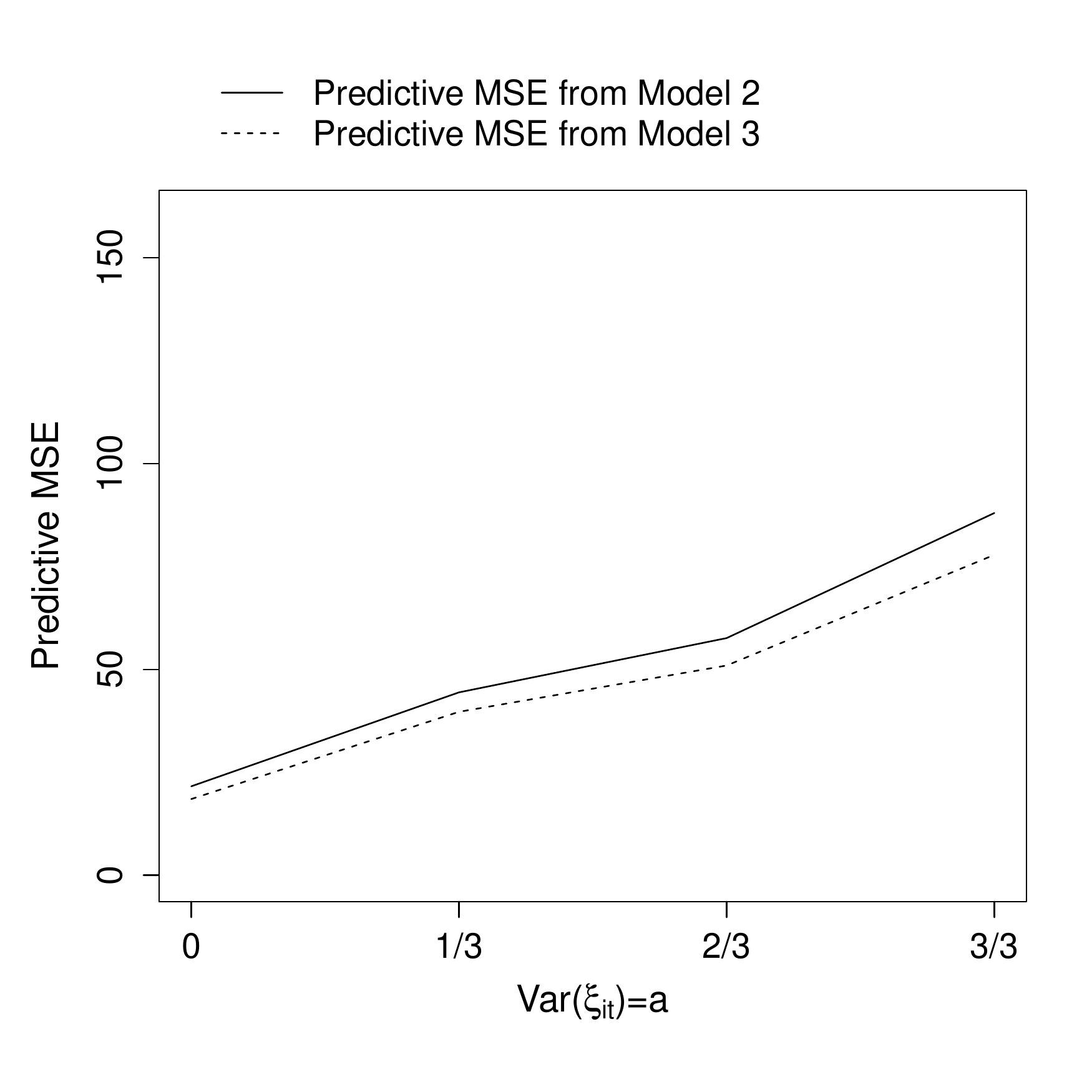}
    }
    \hfill
    \subfloat[Scenario 2 ($\sigma^2=1/6$)]{%
      \includegraphics[width=0.49\textwidth]{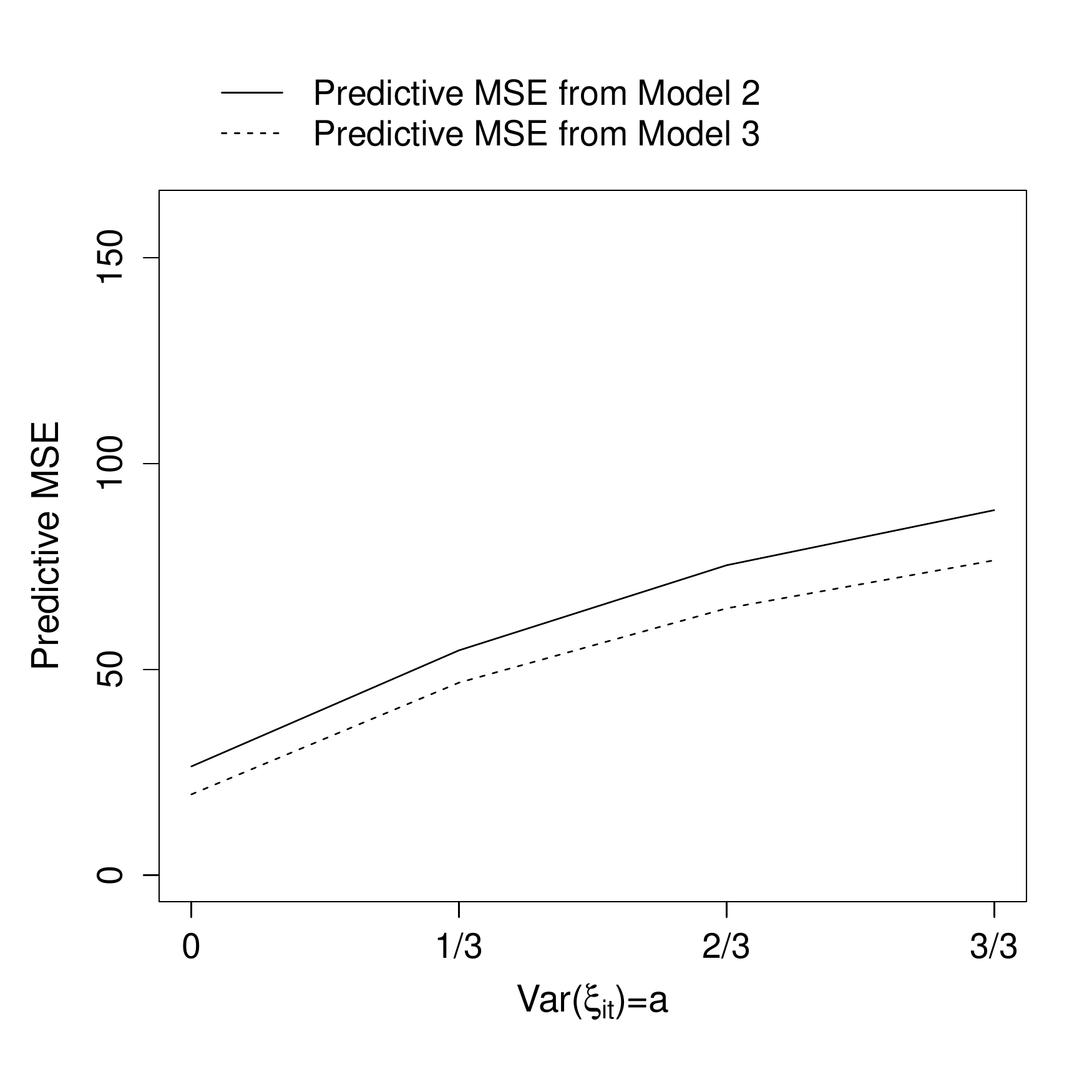}
    }

        \subfloat[Scenario 3 ($\sigma^2=2/6$)]{%
      \includegraphics[width=0.49\textwidth]{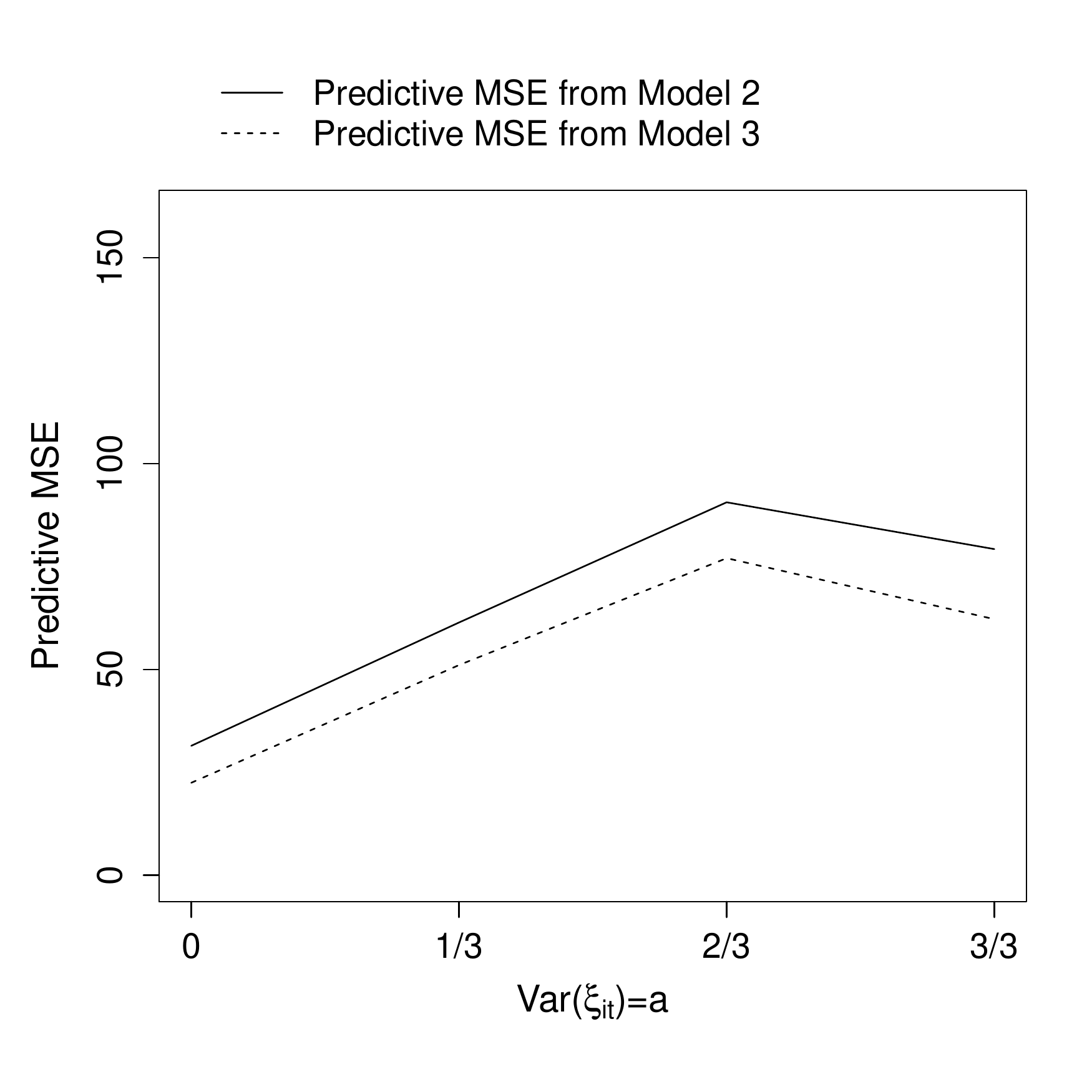}
    }
    \hfill
    \subfloat[Scenario 4 ($\sigma^2=3/6$)]{%
      \includegraphics[width=0.49\textwidth]{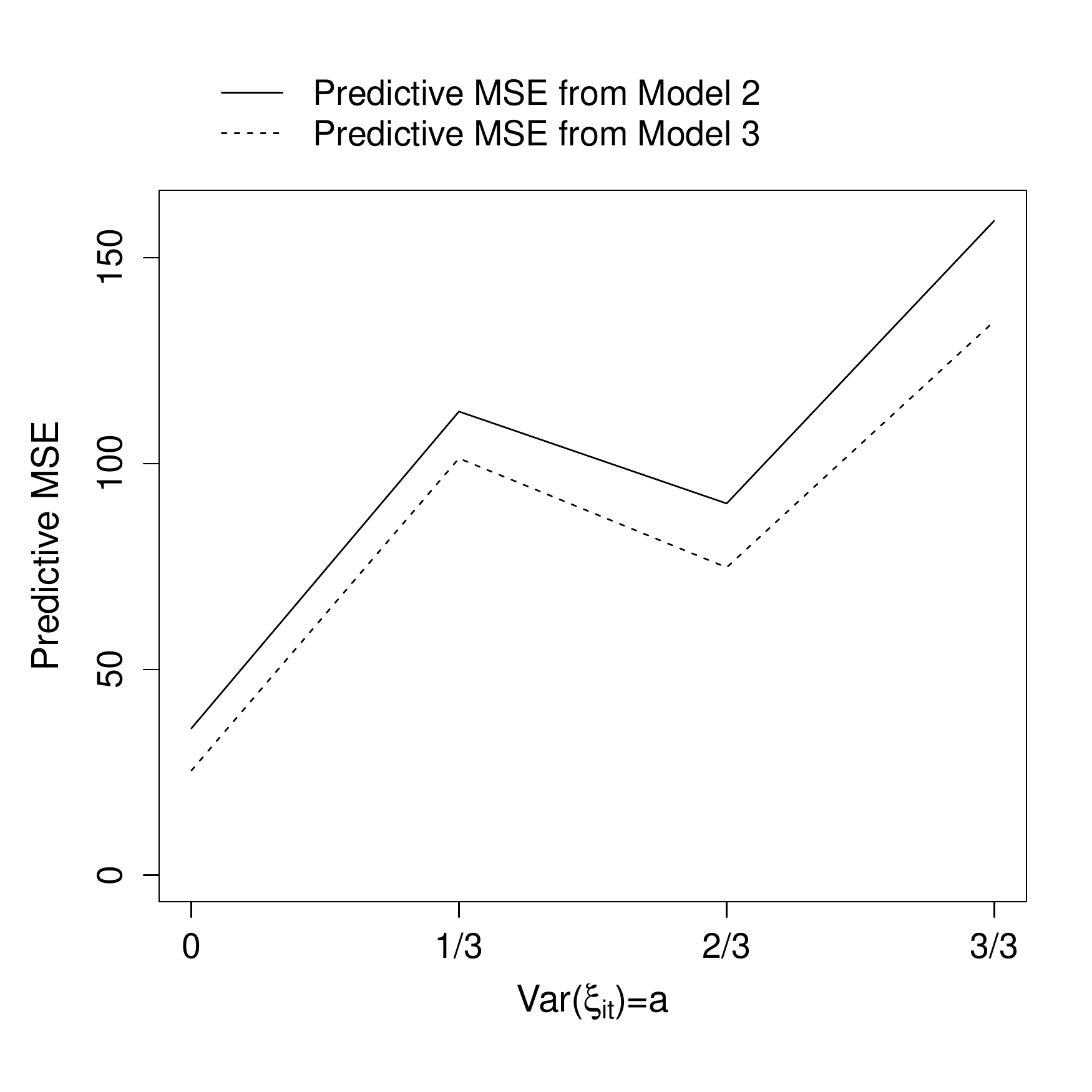}
    }
      \caption{Estimation of Predictive MSE when observations are from Alternative Model \ref{mod.4}.}
     \label{fig.7}
  \end{figure}

%
%
%

\end{document}